\documentclass[11pt]{amsart}

\usepackage{amssymb, amsthm,amstext,amsfonts,amsmath,graphics,lipsum}
\usepackage[normalem]{ulem}
\usepackage{tikz-cd} 
 \usepackage{float}
 \usepackage{graphicx}
 \usepackage{color}
 \usepackage{transparent}
\usepackage[margin=1.25in]{geometry}
\usepackage{marginnote}
\usepackage{url}

\usepackage[latin1]{inputenc}
\usepackage[all]{xy}

\usepackage{tikz}
\usetikzlibrary{matrix}

\makeatletter

\newtheorem*{rep@theorem}{\rep@title}
\newcommand{\newreptheorem}[2]{%
\newenvironment{rep#1}[1]{%
 \def\rep@title{#2 \ref{##1}}%
 \begin{rep@theorem}}%
 {\end{rep@theorem}}}

\makeatother
\newenvironment{dimoclaim}{\emph{Proof of Claim:}\;}{\hfill$\square$}

\newtheorem{theorem}{Theorem}
\newtheorem*{theorem*}{Theorem}
\newreptheorem{theorem}{Theorem}
\numberwithin{theorem}{section}
\newtheorem{theorem-definition}[theorem]{Theorem-Definition}
\newtheorem{proposition}[theorem]{Proposition}
\newtheorem*{proposition*}{Proposition}

\newcommand{\bpfc}{\begin{dimoclaim}}
\newcommand{\epfc}{\end{dimoclaim}}

\newtheorem*{acknowledgements*}{Acknowledgements}

\newtheorem{corollary}[theorem]{Corollary}
\newtheorem*{corollary*}{Corollary}
\newreptheorem{corollary}{Corollary}
\newtheorem{lemma}[theorem]{Lemma}
\newtheorem*{lemma*}{Lemma}

\theoremstyle{definition}
\newtheorem{definition}[theorem]{Definition}
\newtheorem*{definition*}{Definition}

\newtheorem*{claim*}{Claim}
\newtheorem{note}[theorem]{Remark}
\newtheorem*{note*}{Remark}
\newtheorem{question}[theorem]{Question}
\newtheorem{remark}[theorem]{Remark}
\theoremstyle{definition} 
\theoremstyle{remark}
\numberwithin{equation}{section}
\DeclareMathOperator{\Vol}{Vol}

\DeclareMathOperator{\Isom}{Isom}

\DeclareMathOperator{\SL}{\widetilde{SL_2}}
\newcommand{\Z}{\mathbb{Z}}

\newcommand{\fun}[3]{#1 \colon #2 \to #3}

\newcommand{\set}[1]{\left\{#1\right\}}					%set, curly brackets

\DeclareMathOperator{\eqdef}{\doteq\,} %definition

\renewcommand{\tilde}{\widetilde}
\renewcommand{\emptyset}{\varnothing}

\makeatletter
\newcommand{\tpitchfork}{%
  \vbox{
    \baselineskip\z@skip
    \lineskip-.52ex
    \lineskiplimit\maxdimen
    \m@th
    \ialign{##\crcr\hidewidth\smash{$-$}\hidewidth\crcr$\pitchfork$\crcr}
  }%
}
\makeatother

\usepackage{xcolor} 
 \usepackage{tcolorbox}
\newtcbox{\hl}[1][red]{on line, arc=7pt,colback=#1!10!white,colframe=#1!50!black,
  before upper={\rule[-3pt]{0pt}{10pt}},boxrule=1pt, boxsep=0pt,left=6pt,
  right=6pt,top=2pt,bottom=2pt}

 \usepackage[numbers]{natbib}

\DeclareFontFamily{U}{mathx}{\hyphenchar\font45}
\DeclareFontShape{U}{mathx}{m}{n}{<-> mathx10}{}
\DeclareSymbolFont{mathx}{U}{mathx}{m}{n}
\DeclareMathAccent{\widebar}{0}{mathx}{"73}
  \renewcommand{\tilde}{\widetilde}
  \renewcommand{\bar}{\widebar}

 \newtheorem*{Theorem*}{Theorem}

\begin{document}
	 
\title[Hyperbolicity of link complements in Seifert-fibered spaces ]{Hyperbolicity of link complements in Seifert-fibered spaces}
\author[T. Cremaschi and J. A. Rodriguez-Migueles ]{TOMMASO CREMASCHI  AND JOSE ANDRES RODRIGUEZ-MIGUELES}
\thanks{The first author gratefully acknowledges support from the U.S. National Science Foundation grants DMS 1107452, 1107263, 1107367 "RNMS: Geometric structures And Representation varieties" (the GEAR Network) and also from the grant DMS-1564410: Geometric Structures on Higher Teichm\"uller Spaces. The second author was generously supported by the Academy of Finland project \# 297258 "Topological Geometric Function Theory"}

\maketitle
\paragraph*{\textbf{Abstract:}} Let $\bar\gamma$ be a link in a Seifert-fibered space $M$ over a hyperbolic $2$-orbifold $\mathcal O$ that projects injectively to a filling multi-curve of closed geodesics $\gamma$ in $\mathcal O.$ We prove that the complement $M_{\bar\gamma}$ of $\bar\gamma$ in $M$ admits a hyperbolic structure of finite volume and give combinatorial bounds of its volume.

\section{Introduction}
 Let $\Sigma$ be a hyperbolic surface of finite type. In the projective unit tangent bundle $PT^1(\Sigma)$ there is a very special family of links $\widehat\gamma\eqdef(\gamma,\dot\gamma)$ coming from canonical lifts of a geodesic multi-curve $\gamma$ in $\Sigma$. These links correspond to the image under the map ${T^1\Sigma}\rightarrow{PT^1(\Sigma})$ of a collection of periodic orbits of the geodesic flow. Foulon and Hasselblatt  \cite{FH13} gave a topological criterium, depending only on the immersion of $\gamma$ in $\Sigma$, that guarantees the existence of a complete hyperbolic metric of finite volume in the canonical lift complement of $\gamma$ in $PT^1(\Sigma)$.

\begin{theorem}[Foulon-Hasselblatt,\cite{FH13}]\label{FH}
Let $\gamma$ be a closed geodesic on a hyperbolic surface $\Sigma.$ Then, the complement of the canonical lift admits a finite volume complete hyperbolic structure  if and only if $\gamma$ is filling.
\end{theorem}

In \cite{FH13} the previous theorem was stated in a more general setting. The authors considered any embedded lift $\bar\gamma$ in the unit tangent bundle of the hyperbolic surface as long as the projection was injective outside the double points of the closed geodesic $\gamma.$ 
After reading their proof carefully, we noticed that an argument relative to the atoridality of these knot complements was only stated for the particular case of knots coming from periodic orbits of the geodesic flow; on the other hand, the arguments for the other cases worked in greater generality.  

This paper aims to prove the missing argument for the atoroidality of these link complements. This question was posed in a beautiful blog-post of Calegari \cite{Cal} where he gives a geometric proof of Theorem \ref{FH}. We also extend results of the second author from the unit tangent bundle to this setting. Moreover, we  give sequences of closed filling geodesics $\set{\gamma_n}_{n\in\mathbb N}$ in $\Sigma$ and topological lifts $\set{\bar\gamma_n}_{n\in\mathbb N}$ in $PT^1(\Sigma)$ whose associated knot complement volume is bounded linearly in terms of the self-intersection number of the closed geodesic.

\vskip 0.2cm

One of the steps of the proof of the hyperbolicity of $M_{\widehat\gamma}\eqdef PT^1(\Sigma) \setminus\widehat\gamma$ is to show that no essential torus $T\subset M_{\bar\gamma}$ is null-homotopic in $ PT^1(\Sigma) $. To do so, the authors of \cite{FH13} argue that since the geodesic flow is product covered in the universal cover $\widetilde{ PT^1(\Sigma) }$ the complement of all the lifts  $\set{\tilde{\widehat\gamma}}$ of $\widehat\gamma$ is homeomorphic to $\left(\mathbb R^2\setminus X\right)\times\mathbb R$, for $X$ a discrete set. Since $\pi_1\left(\left(\mathbb R^2\setminus X\right)\times\mathbb R\right)$ is free and the essential torus $T$ lifts to $\widetilde{ PT^1(\Sigma) }\setminus\set{\tilde{\widehat\gamma}}$ we reach a contradiction. This is because a free group does not contain any $\Z^2$ subgroup. To avoid using the geodesic flow we will directly show that $\pi_1(\widetilde{ PT^1(\Sigma) }\setminus \set{\tilde{\bar\gamma}})$ is free for any lift $\bar\gamma$ in $PT^1(\Sigma)$ of a geodesic multi-curve on $\Sigma$:

\begin{theorem}\label{uc}Let $M$ be a Seifert-fibered space over a hyperbolic surface $\Sigma$. Let $\bar\gamma$ be a link in $M$ projecting injectively to a filling multi-curve $\gamma\subset \Sigma$ of closed geodesics. Let $q: \widetilde M \rightarrow M$ the universal covering map of $M$ and $\set{\tilde{\bar\gamma}}$ the total preimage of the link $\overline\gamma$ under $q.$ Then, the group $\pi_1\left(\widetilde M\setminus \set{\tilde{\bar\gamma}}\right)$ is free.
\end{theorem}

By adding our argument to their proof we obtain a version of Theorem \ref{FH}  in the more general setting of link complements in  Seifert-fibered spaces, whose projection to their hyperbolic 2-orbifold base  is a filling geodesic multi-curve. Our main result is:

\begin{theorem}\label{sf}
Suppose $\mathcal{O}$ is a hyperbolic $2$-orbifold and $\overline\gamma$ a link in an orientable Seifert-fibered space $M$ over the orbifold $\mathcal{O}$ projecting injectively to a filling geodesic multi-curve $\gamma$  in $\mathcal{O}.$ Then, the complement of $\overline\gamma$ in $M,$ denoted by $M_{\overline\gamma},$ is a hyperbolic manifold of finite volume.
\end{theorem}

Once the hyperbolicity of $M_{\bar\gamma}$ is settled, by the Mostow's Rigidity Theorem \cite{BP1992}, we can pursue the problem of estimating the volume of $M_{\bar\gamma}.$ The volume invariant has been studied in the particular case of canonical lifts of geodesics in  the projective unit tangent bundle $PT^1(\Sigma)$ of a hyperbolic surface $\Sigma$ or the modular orbifold. Upper bounds have been found in terms of the geodesic length in \cite{BPS16} and a combinatorial  lower bound by the second author in \cite{Rod17}.

In \cite[Sec. 5]{Rod17} the second author noticed that the behaviour of the volume of $M_{\bar\gamma}$ among different lifts of $\gamma$ does not depend on the diagram given by the couple $(\gamma,\mathcal O).$ More precisely the second author proved that:

\begin{proposition}\label{lin}
For any hyperbolic metric $X$ on $\Sigma,$ there exists a sequence of $\{\gamma_n\}_{n\in\mathbb N}$ filling closed geodesics and respective lifts $\{\overline{\gamma_n}\}_{n\in\mathbb N}$ in  $PT^1(\Sigma)$ with $\ell_X(\gamma_n)\nearrow \infty,$ such that,
$$ k_X\ell_X(\gamma_n) \leq  \Vol(M_{\overline{\gamma_n}}),$$
where $k_X$ is a positive constant that depends on the metric $X.$ Moreover, there exists a constant $V_0>0$ such that $\Vol(M_{\widehat{\gamma_n}})< V_0$ for every $n\in\mathbb{N},$ where $\widehat{\gamma_n}$ is the canonical lift of $\gamma_n$ on  $PT^1(\Sigma).$
 \end{proposition}
By constructing a particular ideal triangulation on $M_{\bar\gamma}$ one can give a volume upper bound to $M_{\bar\gamma}$, independent of the lift $\bar\gamma$, which is linear in terms of the self-intersection number of $\gamma.$ 
\begin{theorem}\label{ub}
Let $M$ be a Seifert-fibered space over a hyperbolic $2$-orbifold $\mathcal{O}$. Then, for any link $\bar\gamma\subset M$ projecting injectively to a filling geodesic multi-curve $\gamma$ on $\mathcal{O}$:
$$\Vol(M_{\overline{\gamma}})< 8v_3 i(\gamma,\gamma).$$
Where $v_3$ is the volume of the regular ideal tetrahedron and $i(\gamma,\gamma)$ the self-intersection number of $\gamma.$
 \end{theorem}
Furthermore, by \cite[Thm. 1.1]{HP2018} one can construct a continuous lift inside the  projective unit tangent bundle of a punctured hyperbolic surface over some closed geodesics such that the knot complement's hyperbolic volume is, up to a multiplicative factor, the self-intersection number of the geodesic multi-curve. The sequences of geodesics, lifts and estimate of the volume's lower bound of the corresponding knot complements is proven in the following result:

\begin{corollary}\label{how-pur}
Let $\Sigma_{g,n}$ be an $n$-punctured hyperbolic surface, $n\geq 1$, then there exists a sequence of $\{\gamma_n\}_{n\in\mathbb N}$ filling closed geodesics with $i(\gamma_n,\gamma_n)\nearrow \infty,$ and respective lifts $\{\overline{\gamma_n}\}_{n\in\mathbb N}$ in  $PT^1(\Sigma_{g,n})$ such that,
$$\frac{v_8}{2}(i(\gamma_n,\gamma_n)-(2-2g))\leq\Vol(M_{\overline{\gamma_n}})< 8v_3 i(\gamma_n,\gamma_n),$$
where $v_3$ $(v_8)$ is the volume of the regular ideal tetrahedron (octahedron) and $i(\gamma_n,\gamma_n)$ is the self-intersection number of $\gamma_n.$
 \end{corollary}
The previous result shows that self-intersection is the optimal bound when considering general topological lifts. By generalising arguments of the second author \cite[Theorem 1.5]{Rod17} to the Seifert-fibered setting we also give a combinatorial lower bound:

\begin{theorem}\label{1}
 Given a pants decomposition $\Pi$ on a hyperbolic $2$-orbifold $\mathcal O$, a Seifert-fibered space $M$ over $\mathcal O,$ and a filling geodesic multi-curve $\gamma$ on $\mathcal O,$ for any closed continuous lift $\bar\gamma$ we have that : 
$$  \frac{v_3}{2}\sum_{P \in \Pi}(\sharp\{\mbox{ isotopy classes of} \hspace{.2cm}  \bar\gamma\mbox{-arcs in} \hspace{.2cm} p^{-1}(P)\}-3)\leq\Vol(M_{\bar\gamma}),$$
where $v_3$ is the volume of a regular ideal tetrahedron.
 \end{theorem}
 
  \paragraph{\textbf{Outline:}} In section 2 we recall some basic facts about Seifert-fibered spaces and orbifolds. In section 3 we prove Theorem \ref{sf} and Theorem \ref{uc}. In section \ref{geoinv} by using results in \cite{HP2018} and \cite{Rod17} we prove some volume bounds.
 
 \vspace{0.3cm}

\textbf {Acknowledgments:} The second author would like to express his gratitude to the University of Rennes I and the University of Helsinki for creating
an attractive mathematical environment. The second author also thanks Juan Souto for  discussions on these topics. The first author would like to thank Ian Biringer and Martin Bridgeman for helpful discussions. Both authors would like to express their gratitude to Andrew Yarmola for stimulating conversations. Moreover, we would like to thank the anonymous referee for many helpful comments and suggestions.

\section{Seifert-fibered spaces and orbifolds}

In this section, we recall some known facts about the topology of Seifert-fibered spaces and orbifolds. For more details see \cite{Ja,He}.

\begin{definition}

A compact 3-manifold $M$ is a \emph{Seifert-fibered space} if $M$ is the union of a collection $\set{C_\alpha}_{\alpha\in A}$ of pairwise disjoint simple closed curves called \emph{fibers} such that every fiber $C_\alpha$ has a closed neighbourhood $V_\alpha$ homeomorphic to a solid torus and a covering map $p_\alpha:\mathbb D^2\times \mathbb S^1\rightarrow V_\alpha$ satisfying:

\begin{itemize}
\item[(i)] for all $x\in\mathbb  D^2$ we have that $p_\alpha(\set x\times \mathbb S^1)=C_\beta$ for some $\beta\in A$ so that $V_\alpha$ is a union of fibers;
\item[(ii)] $p^{-1}_\alpha(C_\alpha)$ is connected;
\item[(iii)] the group of covering transformation is generated by $r_{n,m}$ for $n,m$ relatively prime integers such that:
$$ r_{n,m}(re^{i\theta},e^{i\phi})\eqdef (re^{i(\theta+2\frac m n\pi)}, e^{i(\phi+\frac{2\pi}n)})$$
\end{itemize}
If $\vert n\vert=1$ we have that $p_\alpha$ is a homeomorphism and we say that $C_\alpha$ is a \emph{regular fiber}, otherwise we say it is a \emph{singular 
fiber}.
\end{definition}

Note that whenever $\vert n\vert> 1$ by (ii) $C_\alpha=p_\alpha(\set 0\times\mathbb  S^1)$ and for $x\neq 0$ we have that $p(\set x\times \mathbb  S^1)$ is mapped to a fiber $C_\beta$ which crosses the meridional disk $p_\alpha(\mathbb D^2\times\set 1)$ $n$ times and wraps $m$ times around $C_\alpha$. Also, since every fiber in a neighbourhood of a singular fiber is regular we get that if $M$ is compact it has finitely many singular fibers.

\begin{definition}
We say that a Hausdorff topological space $\mathcal O$ is an \emph{orbifold} if we have a covering $\mathcal U\eqdef \set{U_i}_{i\in\mathbb N}$, closed under finite intersections, and continuous maps: $\phi_i: V_i\rightarrow U_i$, for $V_i$ open subsets of $\mathbb R^2$, invariant under a faithful linear action of a finite group $\Gamma_i$ such that $\phi_i: V_i/\Gamma_i\rightarrow U_i$ is a homeomorphism. Moreover, we say that the charts $\set{U_i}_{i\in\mathbb N}$ form an \emph{orbifold atlas} if:
\begin{itemize}
\item for $U_i\subset U_j$ we have a monomorphism $f_{ij}:\Gamma_i\hookrightarrow \Gamma_j$;
\item for $U_i\subset U_j$ we have a $\Gamma_i$-equivariant homeomorphism $\psi_{ij}$, called a gluing map, from $V_i$ to an open subset of $V_j$;
\item for all $i,j$ we have $\phi_j\circ\psi_{ij}=\phi_i$;
\item the gluing maps are unique up to compositions with group elements.
\end{itemize}
\end{definition}

\begin{note*} Even though a general orbifold can have  reflections in the rest of this work we will only consider orbifolds with conical points. Therefore, the set of singular points in any orbifold $\mathcal O$ will always be a discrete set.
\end{note*}

If $M$ is a Seifert-fibered space we have a natural projection map: $\pi:M\rightarrow \mathcal O$ obtained by mapping every fiber $C_\alpha$ to a point, the space $\mathcal O$ is called the \emph{orbit-manifold}. Given a neighbourhood of $C_\alpha$ the map $\pi\circ p_\alpha:\mathbb D^2\times\set 1\rightarrow \mathcal O$ is an embedding if $C_\alpha$ is a regular fiber and is equivalent to the projection onto the orbit space of $\mathbb D^2\times\set 1$ under a periodic rotation otherwise. Therefore, the quotient space $\mathcal O$ is naturally an orbifold  with discrete singular locus.

\begin{note*}
From the classification Theorem of Seifert-fibered spaces, see \cite{Se}, follows that any Seifert-fibered space $M$ is homeomorphic to an $\mathbb S^1$-bundle over a compact surface $S$ where we glue some singular neighbourhoods along some tori boundary components. Equivalently, we can think of a Seifert-fibered space as an orientable $\mathbb  S^1$-bundle over a compact orbifold $\mathcal O$.\end{note*}

\section{Hyperbolicity of lift complements}
 The aim of this section is to prove Theorem \ref{sf}:
\begin{reptheorem}{sf}
Suppose $\mathcal{O}$ is a hyperbolic $2$-orbifold and $\overline\gamma$ a link in an orientable Seifert-fibered space $M$ over the orbifold $\mathcal{O}$ that projects injectively to a filling geodesic multi-curve $\gamma$ in $\mathcal{O}.$ Then $M_{\overline\gamma}$ is a hyperbolic manifold of finite volume.
\end{reptheorem}
\begin{definition}
Given a Seifert-fibered space $M$ with its bundle map $p:M\rightarrow \mathcal O$ we say that a link $\bar\gamma\subset M$ \emph{projects injectively} to a multi-curve $\gamma\subset \mathcal O$ if distinct components of $\bar\gamma$ map, under $p$, to distinct components of $\gamma$ and such that the projection $p$ is injective except at self-intersection points of $\gamma$ which have two pre-images.
\end{definition}
Let $M$ be a Seifert-fibered space over a hyperbolic $2$-orbifold $\mathcal O$, $\gamma$ a geodesic multi-curve on $\mathcal O$ and $\overline\gamma$ a link in $M$ projecting injectively to $\gamma$ under $p$. Then we have the following commutative diagram:
$$\xymatrix{ & M \ar[d]^p\\ \coprod_{i=1}^n\mathbb{S}^1 \ar@{^{(}->}[ur]^{\overline\gamma} \ar[r]_\gamma &\mathcal O}$$
From now on we denote by $M_{\overline\gamma}$ the complement of a normal neighbourhood of $\overline\gamma$ in $M.$ 
\begin{definition}
For a hyperbolic $2$-orbifold $\mathcal O$ with a discrete set of singular points $\mathcal S$ we say that a multi-curve $\gamma$ of closed geodesics is \emph{filling} if $\gamma$ is disjoint from $\mathcal S$ and if $\mathcal O\setminus\gamma$ is a collection of disks, once-puncture disks or disks with one conical point.\end{definition}
\begin{figure}[h]
\centering
\includegraphics[scale=0.5]{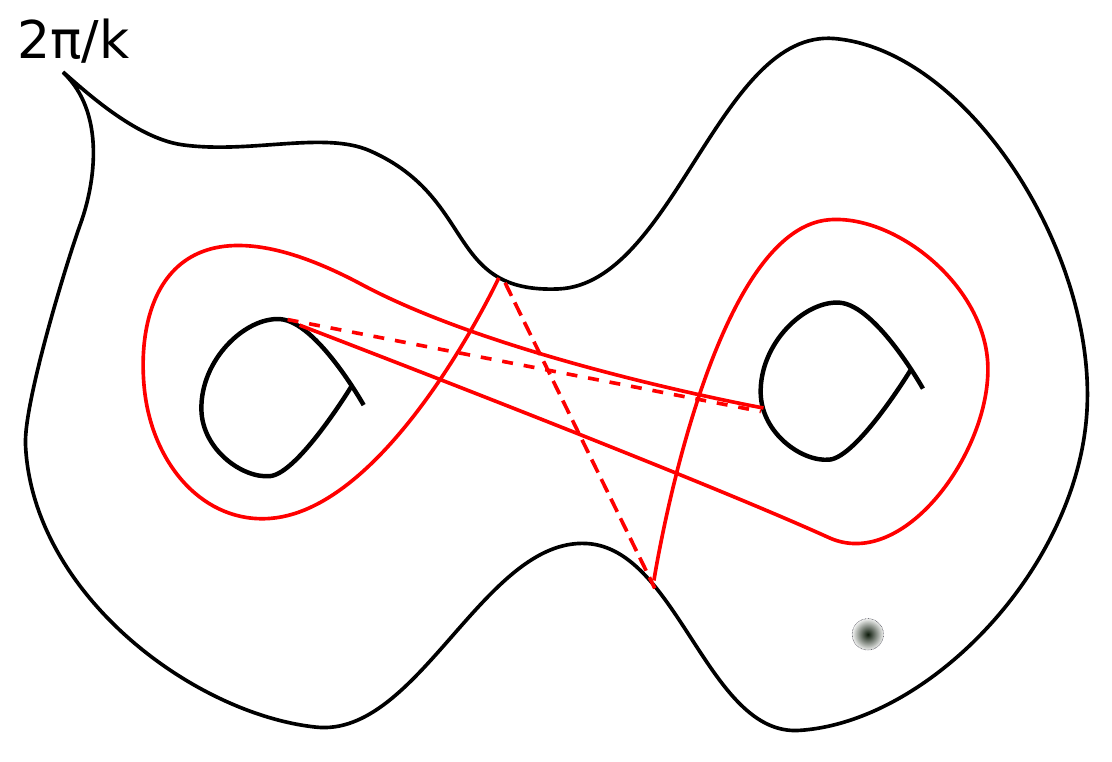}
\caption{A filling geodesic on a $2$-orbifold.
}\label{fill}\end{figure}

In order to prove Theorem \ref{sf} we will first reduce it to the case in which the orbifold $\mathcal O$ is a surface, i.e. to the case where the Seifert-fibered space has a hyperbolic surface $\Sigma$ as base.

\begin{lemma}\label{a}
If $\Sigma\overset q\rightarrow \mathcal O$  is a finite cover from a surface $\Sigma$ and $\gamma$ is a filling geodesic multi-curve in the orbifold $\mathcal O$ then the union of all lifts $\gamma_0$ is also a filling geodesic multi-curve on $\Sigma$.
\end{lemma}

\begin{proof} Let $\Sigma_0\eqdef \mathcal O\setminus \mathcal S$, then $\gamma$ is filling in $\Sigma_0$. Consider the induced cover $q:q^{-1}(\Sigma_0)\rightarrow \Sigma_0$ for $q^{-1}(\Sigma_0)$ a connected subsurface of $\Sigma$. Then $q^{-1}(\gamma)$ is filling in $q^{-1}(\Sigma_0)$. However, $\Sigma=q^{-1}(\Sigma_0)\cup \mathcal D$ for $\mathcal D$ a collection of disks each covering a disk with a cone point. Thus $q^{-1}(\gamma)$ is filling in $\Sigma$ and since $q^{-1}(\gamma)=\gamma_0$ we are done. \end{proof}

\begin{lemma}\label{b}
Given a finite cover $\pi:\widehat M_{\bar{\gamma_0}} \rightarrow  M_{\bar\gamma}$ we have that if $\widehat M_{\bar{\gamma_0}} $ is atoroidal so is $M_{\bar\gamma}$.

\end{lemma}
\begin{proof}

Given an essential torus $ T\subset M_{\bar\gamma}$ the restriction $\pi: \pi^{-1}(T)\rightarrow T$ is a finite cover hence, every component of $\pi^{-1}(T)$ is an essential torus $\widehat T\subset \widehat M_{\bar\gamma_0}$. Thus, since $\widehat M_{\bar\gamma_0}$ is atoroidal the essential torus $\widehat T$ is homotopic into a torus component $\widehat S$ of $\partial \widehat M_{\bar{\gamma_0}}$. The torus component $\widehat  S$ must cover a torus component $S$ of $\partial M_{\bar\gamma}$. By pushing the homotopy via $\pi$ we see that $T$ is also homotopic into a torus component $\pi(\widehat T)$ of $\partial M_{\bar\gamma}$.  \end{proof}

We now reduce the proof of the main theorem to the case in which the orbifold is an actual surface.

\begin{proposition}\label{c} If Theorem \ref{sf} holds for hyperbolic surfaces then it holds for orbifolds.

\end{proposition}
\begin{proof}

By the Geometrization Theorem, see \cite{Sco83}, the Seifert-fibered space $M$ over a compact hyperbolic orbifold $\mathcal O$ has a geometry modelled on either $\mathbb{H}^2\times \mathbb{R}$ or $\SL.$
\vskip .2cm
Assume that $G\cong\pi_1(M)$ is a discrete group of isometries of $\mathbb{H}^2\times \mathbb{R}$ that acts freely and has quotient an orientable $\mathbb{S}^1$-bundle $M.$ Notice that the isometry group of $\mathbb{H}^2\times\mathbb{R}$ can be naturally identified with $ \Isom(\mathbb{H}^2)\times  \Isom(\mathbb{R})$ and we regard the factors as subgroups in the usual way. As $G$ is discrete and $M$ is a Seifert bundle, then $G\cap \Isom(\mathbb{R})=\mathbb{Z}.$ Let $\Gamma$ denote the image of the projection $G\xrightarrow{p} \Isom(\mathbb{H}^2).$ Then we have the exact sequence:
$$1\rightarrow \mathbb{Z}\rightarrow G\rightarrow \Gamma\rightarrow 1.$$
where $\Gamma$ is a discrete group of isometries of $\mathbb{H}^2.$
\vskip .2cm
On the other hand, if $G$ is a discrete group of isometries of $\SL$ acting freely and  with quotient an orientable $\mathbb{S}^1$-bundle $M$ we have the following exact sequence:
$$0\rightarrow \mathbb{R}\rightarrow \Isom(\SL)\xrightarrow{p} \Isom(\mathbb{H}^2)\rightarrow 1.$$
If $\Gamma$ denotes $p(G),$ we have the exact sequence:
$$1\rightarrow \mathbb{Z}\rightarrow G\rightarrow \Gamma\rightarrow 1.$$
for $\Gamma$ a discrete group of isometries of $\mathbb{H}^2.$
\vskip .2cm

In either case $\Gamma$ is a finitely generated subgroup of $\Isom(\mathbb{H}^2),$ thus  by \cite{Bau62} $\Gamma$ is residually finite. Hence, $\Gamma$ has a torsion free subgroup $\widehat\Gamma$ of finite index. Let $\widehat G$ be the subgroup of $G$ projecting onto $\widehat\Gamma$ and let $\widehat M\eqdef\SL/\widehat G$ or $(\mathbb{H}^2\times \mathbb{R})/\widehat G$. By the first isomorphism Theorem \cite{DF} we have that: $G/ \widehat G\cong \Gamma/\widehat \Gamma$ hence $\widehat G$ is also finite index in $G$. Therefore, we have the following commutative diagram: 
$$
\xymatrix{\widehat M \ar[r]^\pi \ar[d]^{\widehat p} & M \ar[d]^{p} \\
\Sigma\eqdef\mathbb{H}^2/\widehat \Gamma \ar[r] ^{\hat\pi}& \mathbb{H}^2/\Gamma\eqdef \mathcal O
}$$
where $\pi:\widehat M\rightarrow M$ is a finite index cover. 
Thus, by lifting $\bar\gamma\subset M$ to $\bar\gamma_0\eqdef\pi^{-1}(\bar\gamma)\subset \widehat M,$ we get a finite cover: 
$$\pi:\widehat M_{\bar{\gamma_0}} \rightarrow  M_{\bar\gamma}$$
Moreover, by the commutativity of the previous diagram the link $\bar\gamma_0$ projects injectively onto the filling multi-curve $\gamma_0=\hat\pi^{-1}(\gamma)$. By Lemma \ref{a} we get that $\gamma_0$ satisfy the conditions of \ref{sf} for $\Sigma.$  Then by Proposition \ref{c} if $\widehat M_{\bar{\gamma_0}}$ is atoroidal, we get that $M_{\bar\gamma}$ is also atoroidal. \end{proof}
Therefore, to show Theorem \ref{sf} it suffices to prove:
\begin{theorem}\label{sfsurface}
Suppose $\Sigma$ is a hyperbolic surface and $\overline\gamma$ is a link in a orientable Seifert-fibered space $M$ over $\Sigma$ that projects injectively to a filling multi-curve $\gamma$ of closed geodesics in $\Sigma.$ Then  $M_{\overline\gamma}$ is a hyperbolic manifold of finite volume.
\end{theorem}

\subsection{Proof of Theorem \ref{sfsurface}}
Before proving Theorem \ref{sfsurface} we need to introduce some objects.
\begin{definition}\label{sim} We say that a triangulation $\tau\eqdef \set{T_i}_{1\leq i\leq m}$ of a hyperbolic surface $\Sigma$  is \emph{simple} for a geodesic multi-curve $\gamma$ if:
 \begin{enumerate}
\item the punctures of $\Sigma$ are contained in the vertices of $\tau$ and every triangle $T\in\tau$ has at most one puncture;
\item the edges and vertices of each element in $\tau$ are distinct;
\item each edge of $\tau$ is a geodesic arc (or geodesic ray if one vertex is a puncture) transversal to $\gamma;$ 
\item in every triangle $T\in\tau$ we have that if $\gamma\cap T\neq\emptyset$ then it contains either two intersecting sub-arcs of $\gamma$ or a single sub-arc of $\gamma$ (see Figure \ref{tria}).
 \end{enumerate}
\end{definition}

 \begin{figure}[h]
\centering
\includegraphics[scale=0.3] {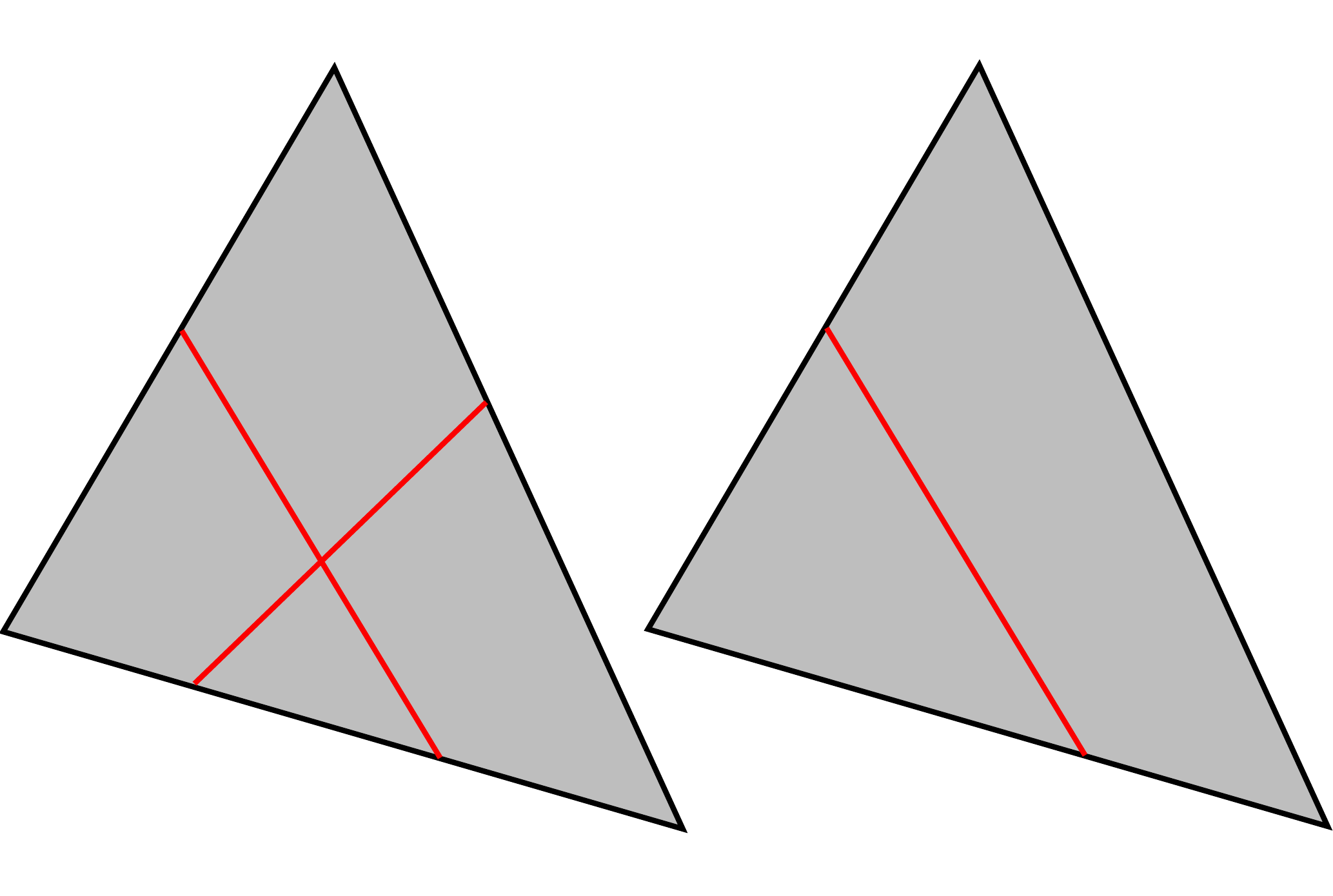}
\caption{Two possible $\gamma$-arcs configuration inside $T.$
}\label{tria}
\end{figure}

\begin{lemma}\label{exsim} Let $\gamma$ be a filling geodesic multi-curve in a hyperbolic surface $\Sigma$. Then, there exist a simple triangulation of $\Sigma$ relative to $\gamma.$
\end{lemma}
\begin{proof} We will build the triangulation in $4$ steps.
 \begin{enumerate}
\item We start our triangulation around the self-intersection points of $\gamma.$ Let $x\in\gamma$ be a self-intersection point and consider a small piece-wise geodesic disk $D$ around it such that it contains only two intersecting $\gamma$-arcs and no punctures. Choose 4 vertices in $\partial D,$ one in each quadrant relative to the pair of $\gamma$-arcs, and the corresponding embedded geodesic quadrangle such that one of the diagonals does not pass through $x.$ The quadrangle, with this diagonal, gives a triangulation $\tau$ around all self-intersection points $x$ of $\gamma.$ Since $\gamma$ is filling, all connected components $C$ of $\Sigma\setminus\gamma$ are punctured disks or disks and they all contain $k>0$ vertices of $\tau$, where $k$ equal to the number of geodesic arcs of $\partial C$. We denote by $\mathcal M$ all components that are monogons or bigons. In either case, every component $C\subset \mathcal M$ is homeomorphic to a punctured disk.

\item Consider a connected component $C$ of $\mathcal M$. If the component of $\Sigma\setminus\gamma$ is a monogon we add an extra vertex $v'$ on the interior so that now we have $3$ marked points in $C$: the ideal vertex, a vertex of $\tau$ and $v'$. Then, we extend $\tau$ as in Figure \ref{simtri}.c).

 For bigons we have already two vertices of $\tau$ and one ideal vertex,  so we extend $\tau$ as in Figure \ref{simtri}.b). We still denote this triangulation by $\tau$ and we note that all components of $\Sigma\setminus \gamma$ that are not triangulated are $n$-gons, with $n\geq 3$, containing $n$ vertices of $\tau$.

\item Let $C$ be a connected $n$-gon, $n\geq 3$, in $\Sigma\setminus\gamma$. Then, we connect the $n$-vertices of $\tau$ in $C$ by geodesic arcs to form a simple loop $\alpha$ isotopic, in $C$ to $\partial C$. If $C$ does not contain an ideal vertex $w$ we add a vertex $v$ and cone the vertices of $\alpha$ to either $v$ or $w$. We still denote this triangulation $\tau$.

\item The components of $\Sigma\setminus\tau$ that have not been triangulated are regular neighbourhoods of $\gamma$-subarcs of the edges of the graph induced by $\gamma$ and can be triangulated by geodesic arcs as in Figure \ref{simtri}.a).
 \begin{figure}[h]
\centering
\includegraphics[scale=0.7] {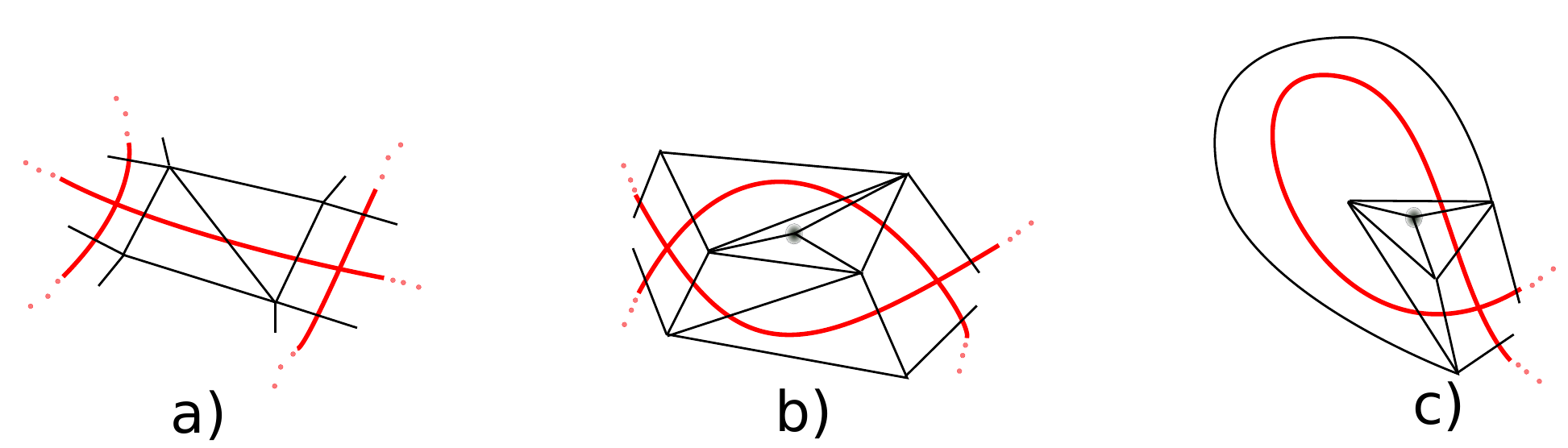}
\caption{Triangulation $\tau$ of $\Sigma$ along an embedded $\gamma$-arc, in red, of a $a)$ $(2+k)$-gon, $b)$ bigon, and $c)$ monogon.
}\label{simtri}
\end{figure}
 \end{enumerate}
Notice that by construction our triangulation satisfies the properties of a simple triangulation in Definition  \ref{sim}.
\end{proof}

 We will now build a partition, relative to a geodesic multi-curve $\gamma$, of the universal cover $\tilde M$ of $M$ induced by a given simple triangulation $\tau$ on $\Sigma$. In our setting we have:
$$\tilde M\overset q\twoheadrightarrow M\overset p\twoheadrightarrow \Sigma$$
Since $p$ is a bundle map we have that for all $T_i\in\tau$, $1\leq i\leq m$, the preimage: $p^{-1}(T_i)$ is homeomorphic to a solid torus $V_{T_i}\cong \mathbb S^1\times \mathbb D^2$. Moreover, the solid torus $V_{T_i}$ inherits from the triangulation $\set{T_i}_{1\leq i\leq m}$ a decomposition of $\partial V_{T_i}$ into:
	 \begin{enumerate}
	 \item three \emph{loops} $w_i^1,w_i^2,w_i^3$ corresponding to the pre-images of the vertices $v_i^1,v_i^2,v_i^3$ of $T_i$;
	 \item three \emph{faces} $F_i^1,F_i^2,F_i^3$ homeomorphic to annuli $I\times \mathbb{S}^1$ and corresponding to the pre-images of the edges $e_i^1,e_i^2,e_i^3$ of $T_i$.
	 \end{enumerate} 
	 By going to the universal cover $\tilde M$ of $M$ each $V_{T_i}$ lifts to a collection: $\coprod_{\alpha\in A_i} T_i^\alpha\times\mathbb R$, with $T_i^\alpha\cong \mathbb D^2$, and the previous decompositions of $V_{T_i}$ induces a decomposition of each $T_i^\alpha\times\mathbb R$ into:

	  \begin{enumerate}
	 \item three \emph{edges} $\tilde w_i^1,\tilde w_i^2,\tilde w_i^3$ each one homeomorphic to $\mathbb R$ and corresponding to the pre-images of the vertices $v_i^1,v_i^2,v_i^3$ of $T_i$;
	 \item three \emph{loops} $\tilde F_i^1,\tilde F_i^2,\tilde F_i^3$ homeomorphic to $I\times \mathbb R$ and corresponding to the pre-images of the edges $e_i^1,e_i^2,e_i^3$ of $T_i$.
	 \end{enumerate} 
\begin{remark}\label{remarka} By the above the discussion using the composition $\tilde M\overset q\twoheadrightarrow M\overset p\twoheadrightarrow \Sigma$ we have the following decomposition of $\tilde M$ into \emph{thick cylinders}:
$$\tilde M=\bigcup_{i=1}^m(p\circ q)^{-1}(T_i)=\bigcup_{i=1}^m\coprod_{\alpha\in A_i} T_i^\alpha\times\mathbb R$$  \end{remark}

\begin{lemma}\label{decomposition}
Let $\tilde M\overset{q}\twoheadrightarrow M$ be the universal covering map of the Seifert-fibered manifold $M$ and $M\overset p \twoheadrightarrow\Sigma$ be the Seifert map, for $\Sigma$ not a sphere. Given any simple triangulation $\tau$ on $\Sigma$ we have that $\tilde M=\cup_{n=1}^\infty K_n$ where each $K_n$ is simply connected and $K_n=K_{n-1}\cup_{S_n} T_{j_n}^{\alpha_n}\times\mathbb R$ for $S_n$ either one or two faces of $T^{\alpha_n}_{j_n}\times\mathbb R$.\end{lemma}
\begin{proof} By remark \ref{remarka} we have the following decomposition of $\tilde M=\bigcup_{i=1}^m\coprod_{\alpha\in A_i} T_i^\alpha\times\mathbb R$. 

\vspace{0.3cm}

 \textbf{Claim 1:} For $i\neq j$ the thick cylinders $T_\alpha^i\times\mathbb R$ and $T_\beta^j\times\mathbb R$ are either disjoint, share at most two faces or share only one edge.

\vspace{0.3cm}

\bpfc Suppose that they are not disjoint so that $T_i=p\circ q (T_\alpha^i\times\mathbb R)$ and $T_j=p\circ q(T_\beta^j\times\mathbb R)$ intersect in $\Sigma.$  Then, since they are distinct elements of the triangulation $\tau$ they must intersect in their boundary. Since $\Sigma$ is not a sphere it follows that $T_i$ and $T_j$ either intersect in a vertex or they share at most two edges and the result follows.
	 \epfc
	 
\vspace{0.3cm}
	 
	 We now claim:
	 
\vspace{0.3cm}

	  \textbf{Claim 2:} There are nested simply connected subsets $\set{K_n}_{n\in\mathbb N}$ of $\tilde M$ such that $\tilde M=\cup_{n=1}^\infty K_n$ and $K_n=K_{n-1}\cup _{S_n} T_{j_n}^{\alpha_n}\times\mathbb R$ where $S_n$ is at most two faces of $T_{j_n}^{\alpha_n}\times\mathbb R$ sharing an edge.
	  
\vspace{0.3cm}

	  \bpfc
	Pick $T_1\in\tau$ and let $K_1$ be any component $T_1^\alpha\times\mathbb R$ of $(p\circ q)^{-1}(T_1)$ in $\tilde M$. Then, mark the edges $\tilde w_1,\tilde w_2,\tilde w_3$ in $\partial K_1$ and \emph{develop} around them. That is, let $\set{T_{i_k}^{\alpha_k}\times\mathbb R}_{1\leq k\leq n_1}$ be the finitely many components of the decomposition of $\tilde M$ containing $\tilde w_1$ as an edge. By Claim 1 at least one of the $\set{T_{i_k}^{\alpha_k}\times\mathbb R}_{1\leq k\leq n_1}$, say $T_{i_1}^{\alpha_k}\times\mathbb R$, shares one or two faces $S$ with $K_1$. We then let $K_2\eqdef K_1\cup_ST_{i_1}^{\alpha_k}\times\mathbb R$. By repeating this for all $\set{T_{i_k}^{\alpha_k}\times\mathbb R}_{1\leq k\leq n_1}$ we have added all solid tori $T_j^\alpha\times\mathbb R$ having $\tilde w_1$ as an edge to $K_1$. The sets $\set{K_n}_{1\leq n\leq n_1+1}$ so constructed are all homeomorphic to $\mathbb D^2\times\mathbb R$, hence simply connected. This is because at every stage we glue a thick cylinder to another thick cylinder along a simply connected subset of their boundary.	 
	 
By repeating this with $\tilde w_2,\tilde w_3$ we get new simply connected subsets $\set{K_n}_{n_1+2\leq n\leq m}$, $m\in\mathbb N$ (see Figure \ref{exh}).

\begin{figure}[h]
\centering
\def\svgwidth{200pt}
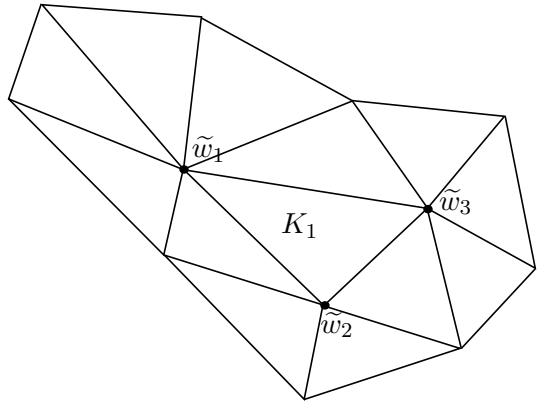
\caption{Schematic of the simply connected subset $K_{m}$ engulfing $K_1$}.\label{exh}
\end{figure}

Moreover, all the $K_n$, $n\leq m$, so constructed are simply connected and properly embedded in $\tilde M$ and $\tilde w_1,\tilde w_2,\tilde w_3$ are contained in the interior of $K_m$. We then mark all edges $\tilde w_1^m,\dotsc \tilde w_{n_m}^m$ of $\partial K_m$ and repeat the previous construction by first adding all the thickened cylinders sharing an edge with $\tilde w_1^m$ and so on.

	 This yields a collection $\set{K_n}_{n\in\mathbb N}$ of properly embedded nested simply connected subset of $\tilde M$ such that for all $n\in \mathbb N:\;\;K_{n+1}= K_n\cup T^{\alpha_{k_n}}_{i_n}\times\mathbb R$, for some $\alpha_{k_n}$, $i_n$.\epfc

Moreover, since each $T_i^\alpha\times\mathbb R$ is in $\cup_{n=1}^\infty K_n$ and the universal cover $\tilde M=\bigcup_{i=1}^m\coprod_{\alpha\in A_i} T_i^\alpha\times\mathbb R$ we have that $\tilde M=\cup_{n=1}^\infty K_n$. 	 \end{proof}
	
We now prove a key Lemma:
\begin{lemma}\label{cyl} Let $M$ be a Seifert-fibered space over $\Sigma$ with projection $p:M\rightarrow \Sigma$ and let $\bar\gamma$ be a link in $M$ such that $\bar\gamma$ projects injectively to a filling geodesic multi-curve $\gamma\subset \Sigma$. Let $\tau$ be a simple triangulation for $\gamma$ in $\Sigma$ and let $\set{\tilde{\bar\gamma}}$ be all the lifts in $\tilde M$ of the link $\overline\gamma.$ Then, for every $T\in \tau,$ we have that $\pi_1\left(\widetilde{p^{-1}(T)}\setminus\set{\tilde{\bar\gamma}}\right) \hspace{.2cm}$ is a free group. Moreover, the set of $\set{\tilde{\bar\gamma}}$-arcs in $\widetilde{p^{-1}(T)}$ forms a free basis.
\end{lemma}

\begin{proof} We claim:
 \begin{claim*}For each $T\in \tau,$ there exists a smooth lift $\overline T$  of $T$ embedded in $p^{-1}(T)\subset M$ such that $\overline\gamma\cap \overline T=\emptyset.$
\end{claim*}

\begin{proof}
Suppose $\gamma_0$  and $\gamma_1$ are two $\gamma$-arcs contained in $T$ with unique intersection point $x\eqdef\gamma_1(t)=\gamma_0(s).$  Consider a new lift $\overline\gamma'_0$ of the arc $\gamma_0$ in $p^{-1}(T)$ which is at positive constant $\mathbb{S}^1$-fiber distance from $\overline\gamma_0$ and does not intersect $\overline\gamma_1.$ Let $\overline\gamma'_1$ be the unique lift passing through $y=\overline\gamma'_0\cap p^{-1}(x)$ and at  $\mathbb{S}^1$-fiber distance $d_{\mathbb{S}^1}(y,\overline\gamma_1(s))$ from $\overline\gamma_1.$ Consider any smooth lift $\overline T$ of $T$ which contains $\overline\gamma'_1\cup \overline\gamma'_0.$ As  $\overline\gamma'_1\cup \overline\gamma'_0$ does not intersect $\overline\gamma$ and the projection of $\overline T\setminus(\overline\gamma'_1\cup \overline\gamma'_0)$ under $p$ is disjoint from $\gamma\cap T$ thus,  $\overline\gamma\cap \overline T=\emptyset.$ The  case of only one $\gamma$-arc inside $T$ follows similarly from the previous case.
\end{proof}

By cutting $p^{-1}(T)$ along the lift $\overline T$ coming from the previous claim we can associate to $p^{-1}(T)\setminus ( \overline T \cup{\overline\gamma})$ a string diagram $D_T$ on $T$ such that $D_T$ has at most one self-crossing. Following Wirtinger, (\cite{Rol76}, Chap.3 Sec.D),  we can give a presentation of the fundamental group $\pi_1(p^{-1}(T)\setminus ( \overline T \cup{\overline\gamma}))$ using $D_T$ and show that: $\pi_1(p^{-1}(T)\setminus( \overline T \cup{\overline\gamma}))$ is a free group. Moreover, the generators are in bijection with the $\overline\gamma$-arcs inside $p^{-1}(T)\setminus \overline T$. Equivalently the bijection is with $\gamma$-arcs inside $T$.

Lastly, since $\widetilde{p^{-1}(T)}\setminus \set{\tilde{\bar\gamma}}$ is obtained by translating one lift of $p^{-1}(T)\setminus (\overline T \cup{\overline\gamma})$ in $\widetilde M$ under the fiber action, meaning that we are gluing consecutive lifts along the common lift of $\overline T$ inside each one of them, and each lift of $\overline T$ is simply connected, by the Van Kampen Theorem, we have that: $\pi_1\left(\widetilde{p^{-1}(T)}\setminus \set{\tilde\gamma'}\right)$ is a free product of free groups. Moreover, by the Van Kampen Theorem the generators are in a $1$-$1$ correspondence with the $\set{\tilde{\bar\gamma}}$-arcs inside $\widetilde{p^{-1}(T)}.$
\end{proof}
We can now prove:
\begin{reptheorem}{uc}Let $M$ be a Seifert-fibered space over a hyperbolic surface $\Sigma$. Let $\bar\gamma$ be a link in $M$ projecting injectively to a filling multi-curve $\gamma\subset \Sigma$ of closed geodesics. Let $q: \widetilde M \rightarrow M$ the universal covering map of $M$ and $\set{\tilde{\bar\gamma}}$ the total preimage of the link $\overline\gamma$ under $q.$ Then, the group $\pi_1\left(\widetilde M\setminus \set{\tilde{\bar\gamma}}\right)$ is free.
\end{reptheorem}

\begin{proof}[\bf{Proof}] 

By Lemma \ref{exsim}, let $\tau$ be a simple triangulation of $\Sigma$ and let $\cup_{i=1}^\infty K_n$ be the induced decomposition of $\tilde M$ coming from Lemma \ref{decomposition}. We define: $C_{T_i^\alpha}\eqdef T_i^\alpha\times\mathbb R \setminus \set{\tilde{\bar\gamma}}$ and let $ C_n \eqdef K_n \setminus \set{\tilde{\bar\gamma}}.$

\vspace{0.3cm}

\textbf{Claim:} For every $n\in \mathbb{N},$ we have that $\pi_1(C_n)$ is free. Moreover, the generators are in bijection with the $\set{\tilde{\bar\gamma}}$-arcs inside $C_n.$

\vspace{0.3cm}

\bpfc The proof is by induction over $n,$ where the base case is Lemma \ref{cyl}. Suppose that the claim is true for $C_n$. We will show it is also true for $C_{n+1}=C_{n}\cup _{Z_n} C_{T_j^\alpha}.$ By Lemma \ref{decomposition} the intersection of $C_n$ with $C_{T_j^\alpha}$ is either one or two punctured faces $Z_n=S_n \setminus \set{\tilde{\bar\gamma}}$ such that each puncture comes from a subset of $\set{\tilde{\bar\gamma}}$-arcs inside $C_n$ and the same holds for $C_{T_j^\alpha}.$ The natural inclusions:
$$(i_1)_*: \pi_1\left(C_{T_j^\alpha}\cap C_n\right)\rightarrow \pi_1\left(C_{T_j^\alpha}\right)\hspace{.2cm}  \mbox{and} \hspace{.2cm} (i_2)_*: \pi_1\left(C_{T_j^\alpha}\cap C_n\right)\rightarrow \pi_1(C_n)  $$
map generators to generators. Thus, by Van Kampen's Theorem $\pi_1(C_{n+1})$ is also free because the new relations are given by:
$$(i_1)_*(s)((i_2)_*(s))^{-1}=1$$
where $s$ is a generating element of $\pi_1\left(C_{T_j^\alpha}\cap C_n\right)$. Thus, we are just pairing the generators of the two free groups. Therefore, the new relations either rename the generators of $\pi_1\left(C_{T_j^\alpha}\right)$ with generators of $\pi_1(C_n)$ or reduce the number of generators of $\pi_1(C_n).$ \epfc

 By the previous claim each $\pi_1(C_n) $ is free and the inclusion $j_n: C_n \rightarrow C_{n+1}$ induce maps $(j_n)_*:\pi_1(C_n)\rightarrow \pi_1(C_{n+1})$ mapping generators to generators. Therefore, the free basis of $C_n$ is extended to a free basis of $C_{n+1}$, thus: 
$$ \varinjlim_{(j_n)_*} \pi_1(C_n)\cong \pi_1\left(\widetilde M\setminus \set{\tilde{\bar\gamma}}\right)$$
is a free group as well. Moreover, the set $\set{\tilde{\bar\gamma}}$ forms a generating set for $\pi_1(\tilde M\setminus\set{\tilde{\bar\gamma}})$.
\end{proof}

We say that a properly embedded arc $\alpha$ in $\mathbb R^3$ is \emph{unknotted} if for any thickened cylinder $V$ such that $\alpha\subset \text{int}(V)$ we have that $\partial V$ is isotopic to $\partial N_\epsilon(\alpha)$ in $V$. As a consequence of the previous proof we obtain:
\begin{corollary}\label{unk}
Given a component $\alpha\in\set{\tilde{\bar\gamma}}$ then $\alpha$ is unknotted in $\widetilde M\cong\mathbb R^2\times\mathbb R$.
\end{corollary}

We can now show:

\begin{reptheorem}{sfsurface}
Suppose $\Sigma$ is a hyperbolic surface and $\bar\gamma$ is a link in an orientable Seifert-fibered space $M$ over $\Sigma$ projecting injectively to a filling multi-curve $\gamma$ of closed geodesics in $\Sigma.$ Then, the complement of $\overline\gamma$ in $M,$ denoted by $M_{\bar\gamma},$ is a hyperbolic manifold of finite volume.
\end{reptheorem}

\begin{proof} By Thurston geometrization Theorem \cite{Thu82} it suffices to show that $M_{\bar\gamma}$ is atoroidal, irreducible and with infinite $\pi_1$. The last two claims follow from standard arguments coming from $3$-dim topology using the fact that $\gamma\neq 0$ in $\pi_1(\Sigma)$ and that $\Sigma$ is not a $2-$sphere, respectively. Thus, since $M_{\bar\gamma}$ is irreducible and $\pi_1(M_{\bar\gamma})$ is infinite we only need to prove the atoroidality condition. The proof involves three cases, each of which will be proven by contradiction. Let $T\subset M_{\bar\gamma}$ be an incompressible torus not parallel to the boundary of $M_{\bar\gamma}$. Then, for $\iota:M_{\bar\gamma}\hookrightarrow M$ the subgroup $\pi_1(\iota(T))$ has either rank zero, one or two in $\pi_1(M)$. 
\vskip .2cm
\textbf{Case 1}: The rank of $\pi_1(\iota(T))$ is zero.
\vskip .2cm
This means that $\iota(T)$ is null-homotopic in $M$. Hence, the map: $\iota:T\rightarrow M$ lifts to an embedded torus $\tilde T$ in $\tilde M$. Moreover, for $\set{\tilde{\bar\gamma}}$ the lifts of $\tilde\gamma\subset M$ we get that $\tilde T\subset \tilde M\setminus\set{\tilde{\bar\gamma}}$ is essential. However, by Theorem \ref{uc} $\pi_1(\tilde M\setminus\set{\tilde{\bar\gamma}})$ is a free group which does not contain any $\mathbb Z^2$ subgroup, giving us a contradiction $\Rightarrow\Leftarrow$.
\vskip .2cm
\textbf{Case 2}: The rank of $\pi_1(\iota(T))$ is one.
\vskip .2cm
If $\text{rank}(\pi_1(\iota(T))=1$ it means that $\iota(T)$ is compressible in $M$. Therefore we have a compression disk $D$ such that compressing $\iota(T)$ along $D$ gives us a 2-sphere $\mathbb S^2\hookrightarrow M$. Since $M$ is irreducible it means that $\mathbb S^2$ bounds a 3-ball $B\subset M$. Thus, we see that $\iota(T)$ bounds a solid torus $V$ in $M$ and by incompressibility of $T$ in $M_{\bar\gamma}$ we must have that $\bar\gamma\cap V\neq\emptyset
$. Since $T\cap \bar\gamma=\emptyset$ we have that every component $\bar\gamma_i\in\pi_0(\bar\gamma)$ intersecting $V$ is contained in $V$.

\vspace{0.3cm}

\textbf{Claim:} There is a unique component $\bar\gamma_i\in\pi_0(\bar\gamma)$ contained in $V$ and it is a generator of $\pi_1(V)$.

\vspace{0.3cm}

\bpfc Let $\alpha$ be a generator of $\pi_1(V)$ in $\pi_1(M)$. Then every component $\bar\gamma_i\subset V$ of $\bar\gamma$ is homotopic, in $V$, to $\alpha^{n_i}$ for some $n_i\in\mathbb N$. But every $\bar\gamma_i$ is the lift of a geodesic in $\Sigma$ and so it is primitive. Hence, every $\bar\gamma_i\subset V$ generates $\pi_1(V)$. Thus, any two $\bar\gamma_i,\bar\gamma_j$ in $V$ must be homotopic contradicting the fact that $\bar\gamma$ projects injectively to a geodesic multi-curve on $\Sigma.$ Thus there is a unique component $\bar\eta\in\pi_0(\bar\gamma)$ contained in $V$ and $[\bar\eta]$ generates $\pi_1(V)$. \epfc

\vspace{0.3cm}

\textbf{Claim:} The torus $T$ is boundary parallel in $M_{\bar\gamma}$.

\vspace{0.3cm}

\bpfc Consider a lift $\widetilde V$ of $V$ in $\widetilde M$. Then $\tilde V$ is homeomorphic to $\mathbb D^2\times \mathbb R$ and it contains $\tilde{\bar\eta}$. If $\widetilde V$ is not boundary parallel in $\tilde M\setminus\set{\tilde{\bar\gamma}}$ we have that the lift $\tilde{\bar\eta}$ is knotted in $\tilde V$ contradicting Corollary \ref{unk}. Therefore, the infinite cylinder $\partial \tilde V$ is isotopic into $\partial N_\epsilon(\tilde{\bar\eta})$. Thus, $\pi_1(T)$ is conjugated into $\pi_1(\partial N_\epsilon (\bar\eta))$ contradicting the fact that $T$ was not parallel to the boundary of $M_{\bar\gamma}$  $\Rightarrow\Leftarrow$.\epfc    

\vspace{0.3cm}

\textbf{Case 3}: The rank of $\pi_1(\iota(T))$ is two.

\vskip .2cm

If $\iota(T)$ is essential in $M$, by Proposition \cite[1.11]{Ha}, we must have that $\iota(T)$ is isotopic to either a horizontal surface or a vertical surface in $M$. If $\iota(T)$ is horizontal it means that the hyperbolic surface $\Sigma$ is covered by a torus which is impossible. Therefore, $\iota(T)$ is isotopic to a vertical torus $T'$. Then if we consider the projection $p:M\rightarrow \Sigma$ we see that $p(T')$ is an essential simple closed curve $\alpha\subset\Sigma$. Moreover, since $T'\cap \tilde\gamma=\emptyset$ we have that $\alpha\cap\gamma=\emptyset$. However, this contradicts the fact that $\gamma$ is a filling multi-curve, giving us a contradiction $\Rightarrow\Leftarrow$.

\vskip .2cm

Thus, $M_{\bar\gamma}$ is atoroidal and hence admits a complete hyperbolic metric of finite volume.\end{proof}
\section{Volume of $M_{\bar\gamma}$}\label{geoinv}

Once the hyperbolicity of $M_{\bar\gamma}$ is settled then by Mostow's Rigidity we can pursue the problem of estimating geometric invariants in terms of topological relations between  the multi-curve $\gamma$ and the hyperbolic orbifold $\mathcal O.$

Specifically we will show our volume upper bounds in terms of self-intersection and extend the lower bound of the second author. Moreover, we will also construct continuous lifts inside the  projective unit tangent bundle of a punctured hyperbolic surface over some closed geodesics such that the knot complement's hyperbolic volume is, up to a multiplicative factor, given by the self-intersection number of the geodesic multi-curve.

\subsection{Lifts $\bar\gamma$ whose volume complement is linear in $\iota(\gamma.\gamma)$}

Recall the following definition:

\begin{definition}
Given a connected, orientable 3-manifold $M$ with boundary we let $S_k(M;\mathbb{R})$ be the \textit{singular chain complex} of $M.$ That is, $S_k(M;\mathbb{R})$ is the set of formal linear combination of $k$-simplices, and we set as usual $S_k(M, \partial M;\mathbb{R})= S_k(M;\mathbb{R})/S_k(\partial M;\mathbb{R})$. We denote by $\|c \|$ the $l_1$-norm of  the $k$-chain $c$. If $\alpha$ is a homology class in $H^{sing}_k(M, \partial M;\mathbb{R})$, the \textit{Gromov norm of $\alpha$}  is defined as
$$\|\alpha\|=\inf_{[c]=\alpha} \{ \|c \|=\sum_{\sigma}|r_\sigma|\hspace{.1cm}\mbox{such that}\hspace{.1cm} c=\sum_{\sigma}r_\sigma \sigma\}.$$
The \textit{simplicial volume} of $M$ is the Gromov norm of the fundamental class of $(M,\partial M)$ in  $H^{sing}_3(M,\partial M;\mathbb{R})$ and is denoted by $\|M\|.$ In the special case in which $\partial M$ has only tori boundary components there is another similar definition of $\| M\|_0$ which coincides with $\| M\|$ whenever $M$ is hyperbolic and has the property that for any Seifert-fiebered space $N$: $\| N\|_0=0$.
\end{definition}
and the following results:
\begin{proposition}{\cite[6.5.2]{Th1978}} Let $(M,\partial M)$ be a compact, orientable 3-manifold with $\partial M$ consisting of tori. If $(N,\partial N)$ is obtained from $M$ by gluing pairs of tori in $\partial M$ then:
$$\|N\|_0\leq \| M\|_0$$
\end{proposition}
\begin{lemma}{\cite[6.5.4]{Th1978}} Let $M$ be a complete hyperbolic manifold of finite volume. Then, $v_3\|M\|=v_3\| M\|_0=\Vol(M)$.
\end{lemma}

Then, we have:

\begin{theorem}\label{lbsi}
Let $\Sigma_{g,n}$ be a hyperbolic surface and $M\cong \Sigma_{g,n} \times \mathbb S^1,$  then there exists a sequence of  filling closed geodesic $\{\gamma_n\}_{n\in\mathbb N}$ with $i(\gamma_n,\gamma_n)\nearrow \infty,$ and respective lifts $\{\overline{\gamma_n}\}_{n\in\mathbb N}$ in  $M$ such that,
$$\frac{v_8}{2}(i(\gamma_n,\gamma_n)-(2-2g))\leq\Vol(M_{\overline{\gamma_n}}),$$
where $v_8$ is the volume of the regular ideal octahedron and $i(\gamma_n,\gamma_n)$ the self-intersection number of $\gamma_n.$
 \end{theorem}

\begin{proof}
For the sake of concreteness, we will  first prove the result for the once-punctured torus $\Sigma_{1,1}.$  Let $\gamma_n$ be constructed as in Figure \ref{alt1}. That is, fix a simple closed geodesic $s$ on $\Sigma_{1,1},$ and pick two distinct points $p_1,p_2$ on $s.$  Let $\alpha_1$, $\beta_1$ be two essential arcs linking in $\Sigma_{1,1}\setminus s$ linking $p_1$ with $p_2$ and such that $\iota(\alpha_1,\beta_1)=1$. Note, that $\alpha_1\cup\beta_1$ gives us an essential loop in $\Sigma_{1,1}$. Let $\gamma_1$ be the closed geodesic representative of $\alpha_1\cup\beta_1$. Then, $\gamma_n$ is obtained from $\gamma_1$ by Dehn-twisting $\beta_1$ $2n$-times along $s.$

\begin{figure}[h!]
\centering
\includegraphics[scale=0.6
] {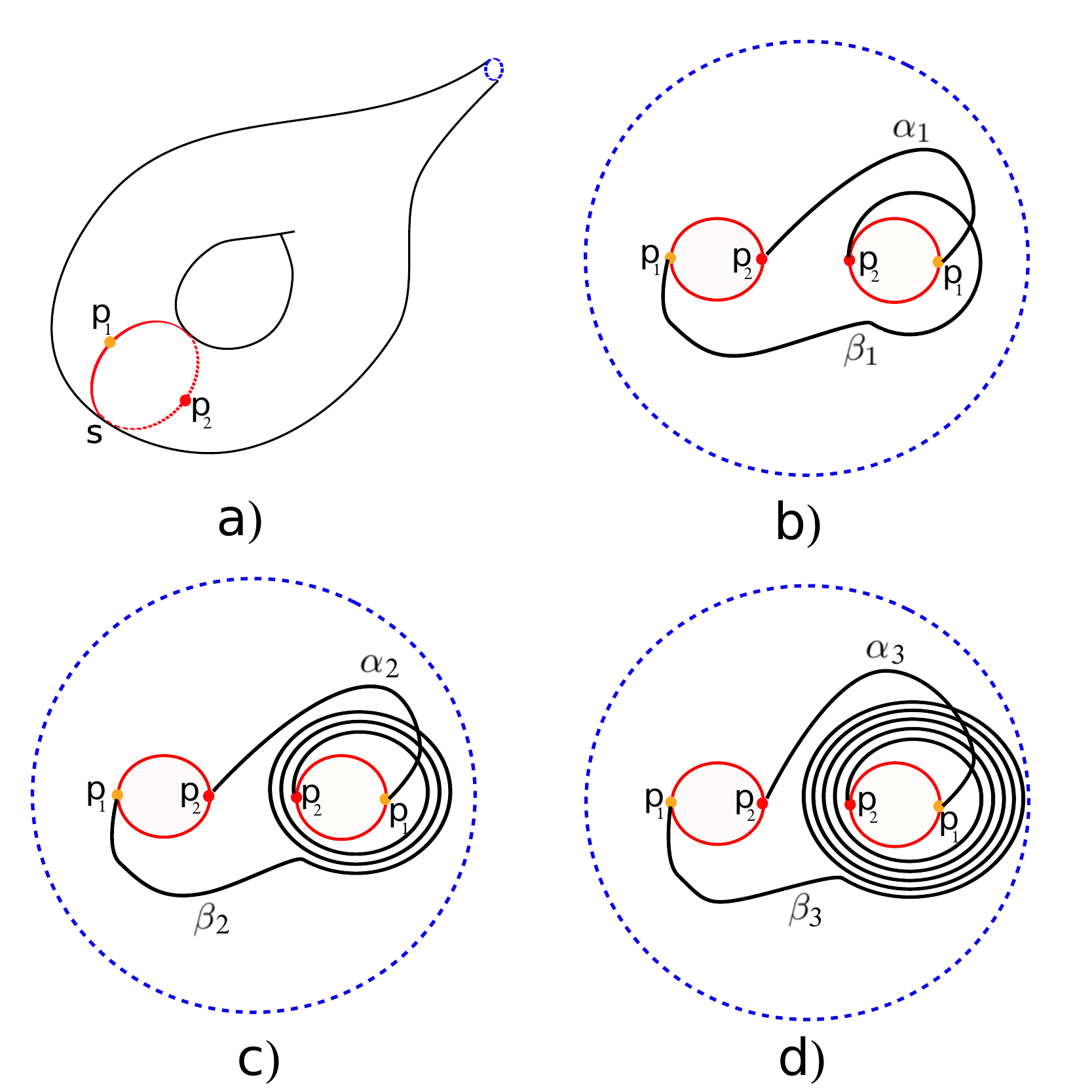}
\caption{a)$s$ on $\Sigma_{1,1},$ b) $\gamma_1,$ c) $\gamma_2,$ and d) $\gamma_3.$
}\label{alt1}
\end{figure}

Since $M\cong \Sigma_{1,1} \times \mathbb S^1,$ consider a global section $S_{1,1}$ embedded in $M.$ Let $\overline{\gamma_n}$ be constructed in $N_\varepsilon(S_{1,1})$, a normal $\varepsilon$-neighbourhood of $S_{1,1},$ such that its corresponding diagram on $S_{1,1}$ is the alternating diagram in Figure \ref{alt2}.

\begin{figure}[h!]
\centering
\includegraphics[scale=0.6
] {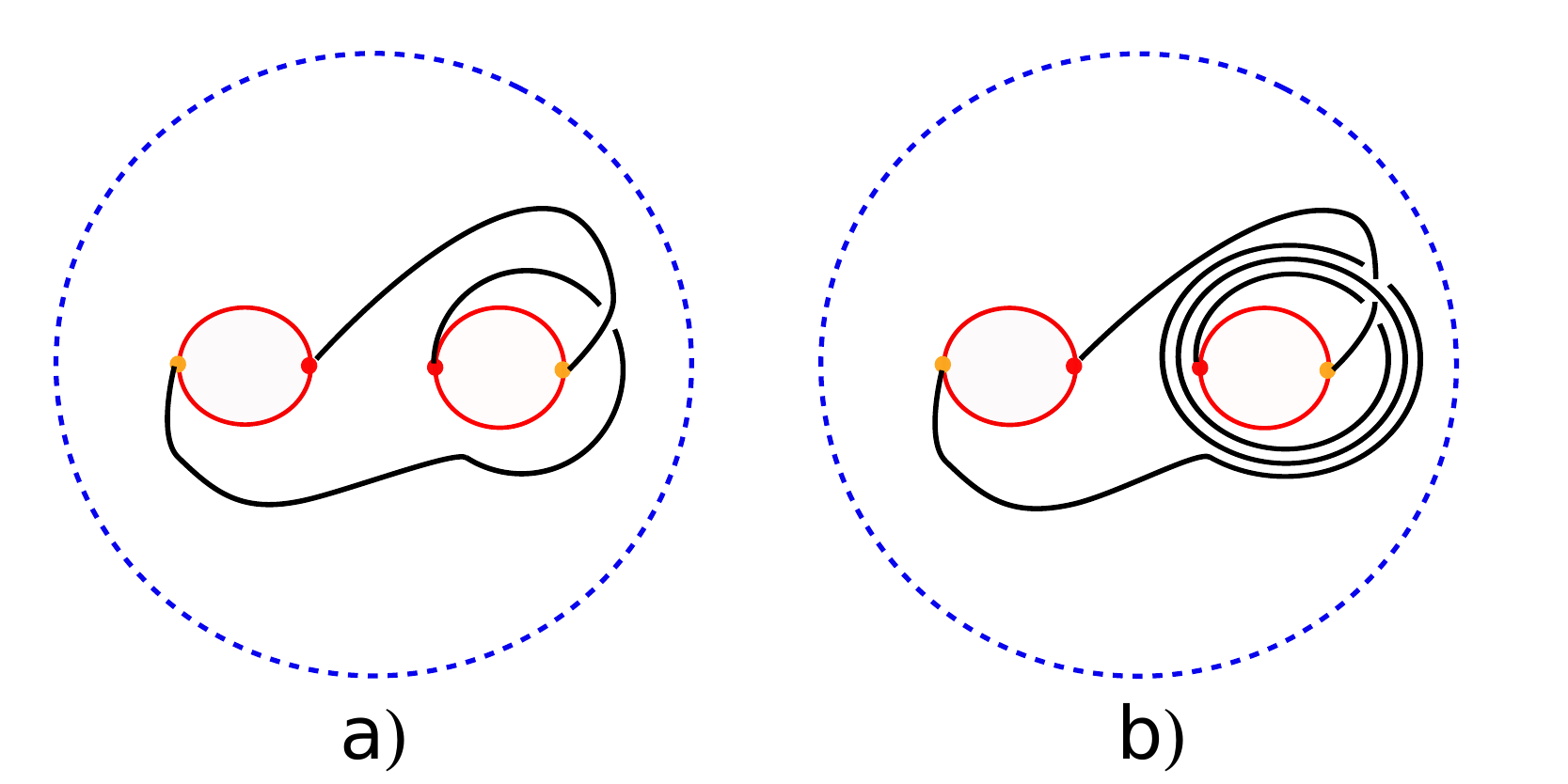}
\caption{An alternating diagram for a) $\gamma_1,$ and b)$\gamma_2.$ 
}\label{alt2}
\end{figure}
By making trivial Dehn filling around the torus coming from the puncture of $\Sigma_{1,1},$ the manifold $(M,S_{1,1})$ becomes $(\Sigma_1\times \mathbb S^1,S_1)$ and the position of our knot $\overline{\gamma_n}$ does not change. By using \cite[6.5.4]{Th1978}, \cite[6.5.2]{Th1978} and the fact that a solid torus $V$ has $\|V\|_0=0$ we get:
$$v_3 \|(\Sigma_1\times \mathbb S^1)_{\overline{\gamma_n}} \|_0 \leq v_3 \|(M_{\overline\gamma_n}) \|_0 =\Vol(M_{\overline{\gamma_n}})$$
where the last equality comes again from \cite[6.5.4]{Th1978}.

Notice that $(\Sigma_1\times \mathbb S^1)_{\overline{\gamma_n}}$ is not hyperbolic, however it contains an hyperbolic piece given by $N_\varepsilon(S_1)\setminus{\overline{\gamma_n}}$. Since, $N_\varepsilon(S_1)\setminus{\overline{\gamma_n}}$ is $(\Sigma_1\times \mathbb S^1)_{\overline{\gamma_n}}$ split along an essential torus by \cite[6.5.2]{Th1978}:
$$\Vol(N_\varepsilon(S_1)\setminus{\overline{\gamma_n}})= v_3 \|N_\varepsilon(S_1)\setminus{\overline{\gamma_n}}\|_0\leq v_3 \|(\Sigma_1\times \mathbb S^1)_{\overline{\gamma_n}} \|_0$$
Furthermore, the projection of $\overline{\gamma_n}$ has a weakly twist-reduced, weakly generalised alternating diagram on a generalised projection surface, see \cite[Sec.2]{HP2018}, $S_1$ in $N_\varepsilon(S_1)$.  Since, $N_\varepsilon(S_1)\setminus N_{ \frac{\varepsilon}{2}}(S_1)$ is atoroidal, and $\partial$-anannular\footnote{This means that it has no essential annulus whose boundary is not contained in the boundary components isotopic to the removed fiber.}, $N_\varepsilon(S_1)$ is boundary incompressible in $N_\varepsilon(S_1)\setminus N_{ \frac{\varepsilon}{2}}(S_1)$, and $\gamma_n$ is filling in $\Sigma_1,$ we can apply \cite[Thm.1.1]{HP2018}:
$$\frac{v_8}{2}(tw(\gamma_n)-\chi(\Sigma_1))\leq \Vol(N_\varepsilon(S_1)\setminus{\overline{\gamma_n}}),$$ 
where $tw$ is the number of twisting regions of the link diagram \cite[Def.6.4]{HP2018}.  In the case of closed geodesics in minimal positions we do not have bigons in its diagram. Therefore, $tw$ is equivalent to the self-intersection number of the corresponding geodesic.
\vskip .2cm

 To generalise this result to any hyperbolic surface $\Sigma_{g,n}$ notice that:

\begin{enumerate}

\item The number of connected components of the sequence of $\{\Sigma_{1,1}\setminus \gamma_k\}_{k\in\mathbb N}$ tends to infinity. Then, the previously constructed sequence has a sub-sequence $\{ \gamma_k\}$ such that $\Sigma_{1,1}\setminus \gamma_k$ has more than $n$ connected components. Then, by removing one puncture in $n$ simply connected components of $\Sigma_{1,1}\setminus \gamma_k$ we can think of $\set{\gamma_k}$ as in $\Sigma_{1,n}$.

\item  It is a straightforward  exercise  to  show that   any   projection of a link  on the $2$-sphere can be made alternating by changing crossings. Then any closed geodesic in $\Sigma_{0,n}$ admits an alternating diagram.

\end{enumerate}
Given $\alpha_1$ and $\alpha_2$ be a filling closed geodesics on $\Sigma_{g,1}$ and $\Sigma_{1,2}$ respectively. Let $\alpha_1\bar\star\alpha_2$ be the closed geodesic homotopic to a closed curve obtained by surgering $\alpha_1$ and $\alpha_2$ along a simple arc meeting transversely one boundary component in each surface, see \cite[Subsec. 4.2]{Rod17}. To prove the $\Sigma_{g,1}$ case with $g\geq2$ we can proceed by induction on the genus, using the following claim:

\vspace{0.3cm}

\textbf{Claim:} Let $\alpha_1$ and $\alpha_2$ be a filling closed geodesics admitting an alternating diagram on $\Sigma_{g,1}$ and $\Sigma_{1,2}$ respectively.  Then $\alpha_1\bar\star\alpha_2$  is filling and admits an alternating diagram on $\Sigma_{g+1,1}=\Sigma_{g,1}\cup_\partial \Sigma_{1,2}.$

\vspace{0.3cm}

\bpfc The filling property is proven in (\cite{Rod17}, Claim. 4.13) and the existence of an alternating diagram follows from fixing an alternating diagram on each $\alpha_1$ and $\alpha_2$.
\begin{figure}[h]
\centering
\includegraphics[scale=0.7
] {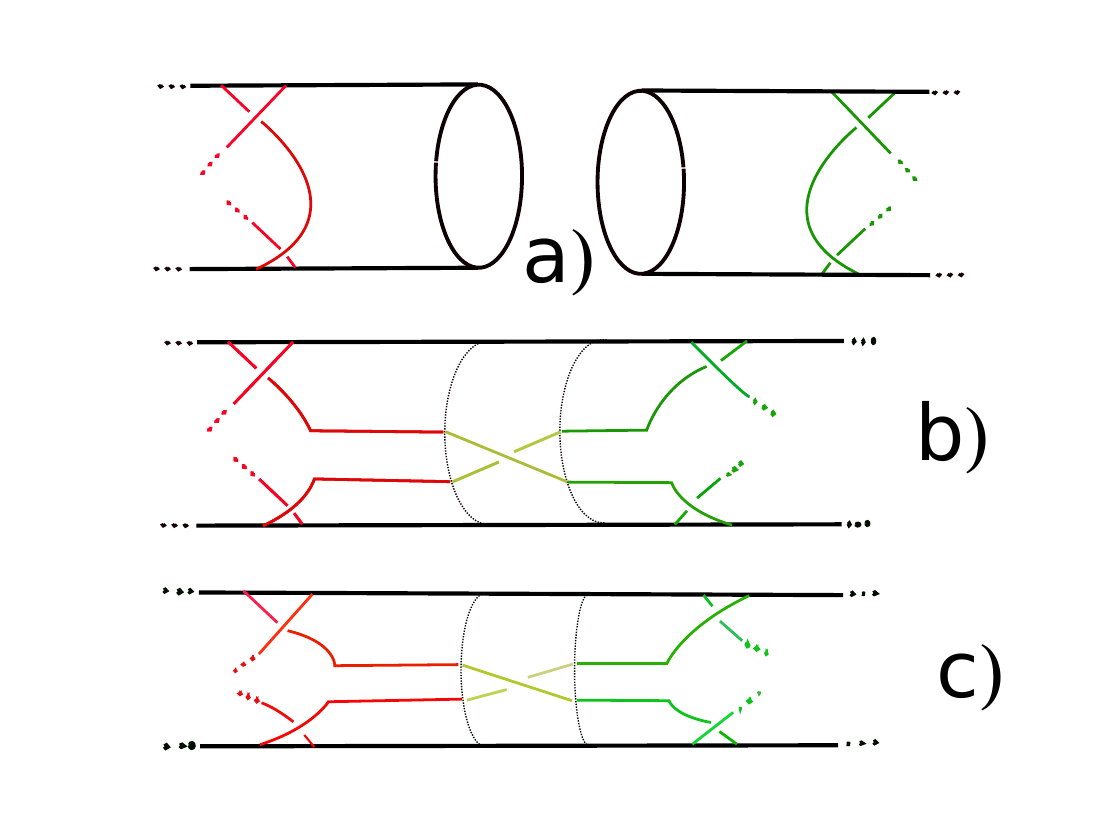}
\caption{a) $\alpha_1$ and $\alpha_2$ in a neighbourhood of the glued boundary component, b) The induced projection diagram of  $\alpha_1\bar\star\alpha_2$ around a  neighbourhood of the glued boundary component, c) Changing to the opposite crossing projection on one of the $\alpha_2$ subarcs (green)  to obtain an alternating diagram.}\label{alt4}
\end{figure}

 If after connecting both geodesics the corresponding diagram is not alternating (see Figure \ref{alt4}) then, by changing the crossing orientation of all crossings in one of the sub-arcs $\alpha_i$ making the diagram of the geodesic corresponding to $\alpha_1\bar\star\alpha_2$ alternating.\epfc
 
Finally to find the sequences of geodesics for general hyperbolic surface $\Sigma_{g,n}$ we use the analog argument used for the case of $\Sigma_{1,n}$ in $(1),$ so that we could add or remove punctures. \end{proof}

\begin{note}

Not every closed geodesic on a surface of genus greater or equal than $1$ admits an alternating diagram (see Figure \ref{alt3}.b). Even though, for each hyperbolic surface, one can find an infinite number of distinct types of closed geodesics which admit an alternating diagram, see Figure \ref{alt3}.c.
\end{note}

\begin{figure}[h]
\centering
\includegraphics[scale=0.5
] {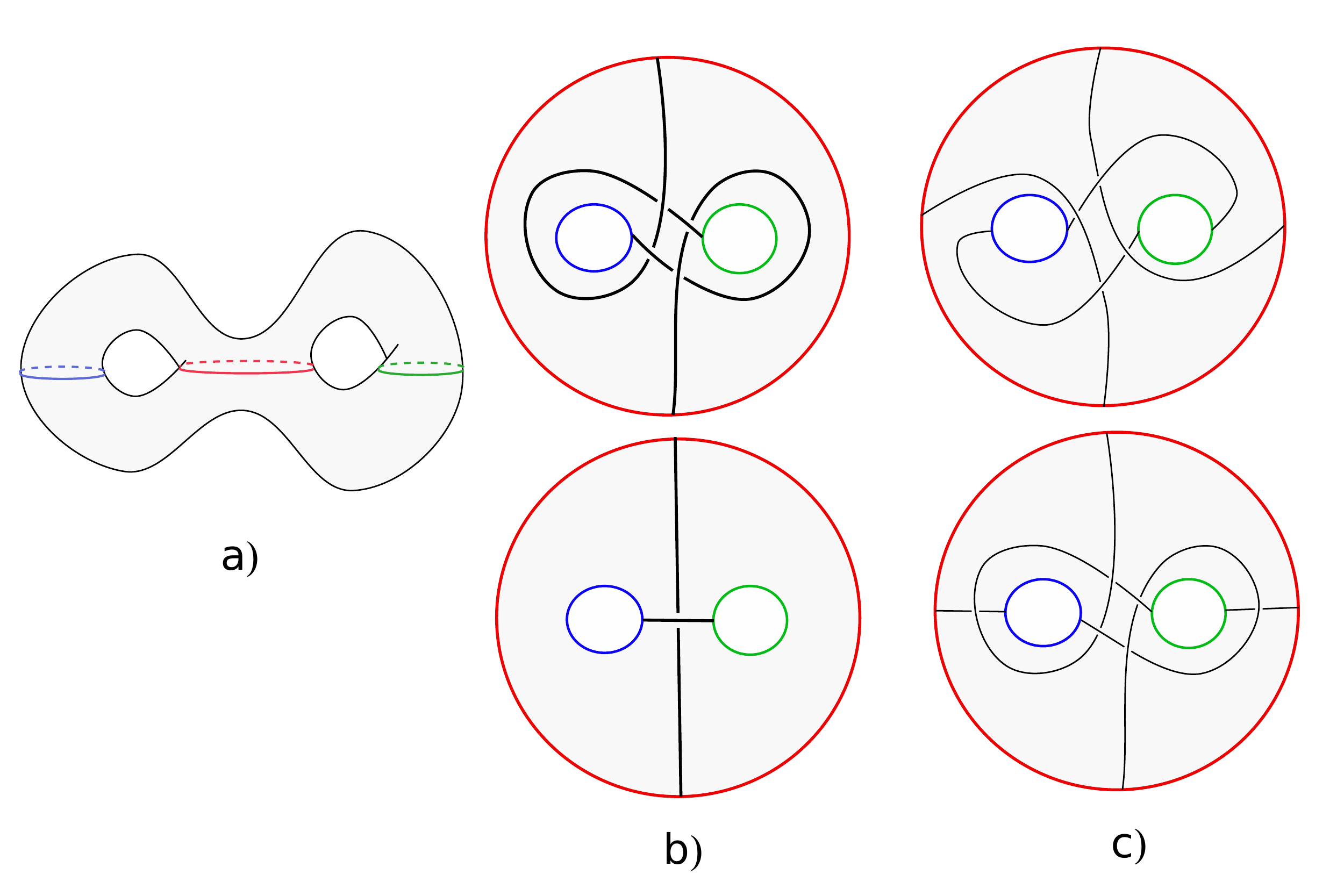}
\caption{a) Pants decomposition on $\Sigma_2,$  b) Closed geodesic not admitting an alternating diagram, c) Closed geodesic with an alternating diagram.}\label{alt3}
\end{figure}

We show now a general volume's upper bound for any lift complement on Seifert-fibered spaces over a filling geodesic multi-curve\string:

\begin{theorem}\label{ub}
Let $M$ be a Seifert-fibered space over a hyperbolic $2$-orbifold $\mathcal{O}$. Then, for any link $\bar\gamma\subset M$ projecting injectively to a filling geodesic multi-curve $\gamma$ on $\mathcal{O}$:
$$\Vol(M_{\overline{\gamma}})< 8v_3 i(\gamma,\gamma).$$
Where $v_3$ is the volume of the regular ideal tetrahedron and $i(\gamma,\gamma)$ the self-intersection number of $\gamma.$
 \end{theorem}

\begin{proof} 
The idea is to build a hyperbolic link $L_{\gamma}$ inside $M$ that reduces the \emph{complexity} of $\bar\gamma,$ in the sense that $M_{\bar\gamma}$ is obtained by performing Dehn filling along some components of $L_{\gamma}$. Since Dehn filling does not increase the volume \cite[ Theorem 6.5.6]{Th1978} and the fact that the number of tetrahedra in any ideal tetrahedra decomposition of a finite volume hyperbolic manifold is an upper bound for its volume \cite[Theorem 6.1.7]{Th1978}, we have that:
$$\Vol(M_{\bar\gamma})\leq  \Vol( M\setminus{L_{\gamma}} ) \leq v_3 \sharp\mathcal{T}_{L_\gamma},$$
where $\mathcal{T}_{L_\gamma}$ is a decomposition of $M\setminus{L_{\gamma}}$  into ideal tetrahedra. That is, the vertices corresponds to the cusps of $M\setminus L_\gamma$. After constructing the link $L_\gamma$, we will argue that there exist $\mathcal{T}_{L_\gamma}$ with the number of tetrahedra comparable to the self-intersection number of $\gamma.$ 

Let $\mathcal F$  be the collection of fibres of $M$ projecting under $p$ to conical points of $\mathcal O$. For every simply connected region $D$ of $\mathcal{O}\setminus\gamma$ not containing a conical point we pick a regular fibre $F_D$ whose projection lies in $D$ and call this collection of fibres $\mathcal D$. Let us denote by $N$ the Seifert-fibered space obtained by removing $\mathcal F\cup\mathcal D$ from $M$. Since $N$ has no singular fibres let $\Sigma$ be the Seifert surface of $N$. Note that, $\Sigma$ is homeomorphic to $\mathcal{O}$ minus the set of conical points and minus one point for each simply connected component of $\mathcal{O}\setminus\gamma.$  Then, we define $L_\gamma\eqdef \bar\gamma \cup\mathcal F\cup\mathcal D$. By construction $N_{\bar\gamma}=M\setminus{L_{\gamma}},$ and by Theorem \ref{sf} it admits a finite volume hyperbolic structure.

To give a decomposition of $N_{\bar\gamma}$  into ideal tetrahedra, we start by taking a pair of ideal vertices in each fibre that projects to a self-intersection points of $\gamma$ such that they connect the two points on $\bar\gamma$ (see Figure \ref{fibertria}). Moreover, let $G_\gamma$ be the $4$-valent graph induced by $\gamma$ on $\Sigma$ and let $A_\alpha\cong \mathbb S^1\times I$, for $\alpha$ an edge of $G_\gamma$ be the pre-image under $p\vert_N$. 

 \vspace{0.3cm}
 \begin{figure}[h]
\centering
\includegraphics[scale=0.5] {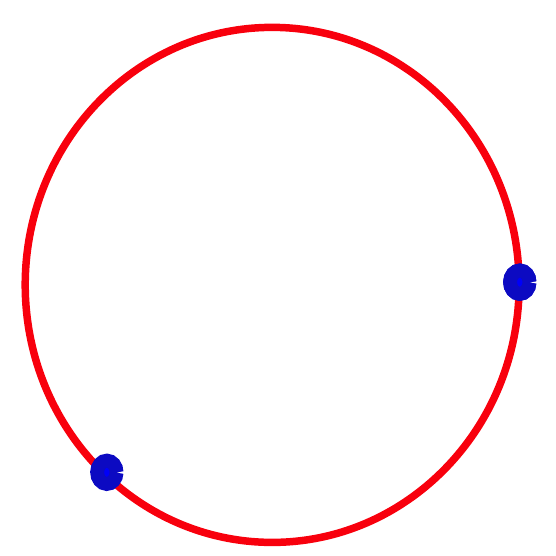}
\caption{Triangulation of the fibres coming from self-intersection points of $\gamma.$
}\label{fibertria}
\end{figure}

We extend this graph to an ideal triangulation of the CW-complex $p\vert_{N}^{-1}(\gamma)=\cup_{\alpha\in E(G_\gamma)}A_\alpha$ by  triangulating each annulus $A_\alpha$. We do this by adding an ideal edge, which is an embedded arc connecting the other vertices in each boundary fibre that do not intersect the embedded $\bar\gamma$-arc in the annulus (up to isotopy this arc is unique) and then collapsing the $\bar\gamma$-arc in the annulus to a point. This induces an ideal triangular decomposition of each annulus by two ideal triangles (see Figure \ref{arctria}).
 \vspace{0.3cm}
 \begin{figure}[h]
\centering
\includegraphics[scale=0.5] {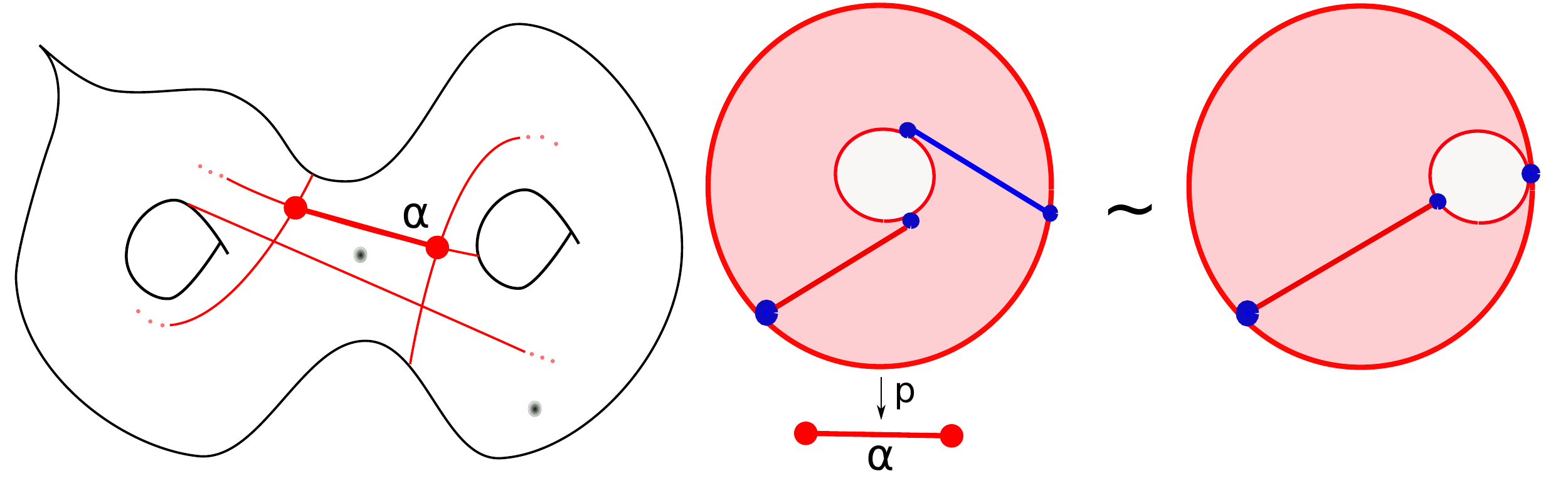}
\caption{Triangular ideal decomposition of $p_{|N}^{-1}(\alpha)$ where the $\bar\gamma$-arc is in blue.
}\label{arctria}
\end{figure}

Let $H\subset N$ be a regular neighbourhood of the triangulated CW-complex $p^{-1}\vert_N(\gamma)$. Then, $H$ has a natural prism-decomposition induced by the ideal triangulation of $p^{-1}\vert_N(\gamma),$ where some vertices correspond to $\bar\gamma.$ Since $\gamma$ fills  $\Sigma$ and we added a puncture in every complementary disk region we have that $N$ is homeomorphic to the interior of $H.$ Moreover, the prism-decomposition of $H$ induces a triangulation of $\partial H,$ which are tori corresponding to fibres of punctures of $\Sigma.$ Therefore, by collapsing the boundary components of $H$ to an ideal vertex we obtain an ideal triangulation of $N_{\bar\gamma},$ because the ideal vertices of our ideal triangulation project precisely to the cusps of $N_{\bar\gamma}.$

Finally, the number of ideal tetrahedra used in this triangulation is four times the number of edges in the graph associated to $\gamma.$ The number of edges is at most two times the self-intersection number of $\gamma$. Hence, we have at most eight ideal tetrahedra for each self-intersection point of  $\gamma.$
\end{proof} 

As a corollary of Theorems \ref{lbsi} and \ref{ub}, and the fact that Seifert-fibered spaces over punctured surfaces are homeomorphic to trivial circle bundles we obtain:

 \begin{repcorollary}{how-pur}
Let $\Sigma_{g,n}$ be an $n$-punctured hyperbolic surface, $n\geq 1$, then there exists a sequence of $\{\gamma_n\}_{n\in\mathbb N}$ filling closed geodesics with $i(\gamma_n,\gamma_n)\nearrow \infty,$ and respective lifts $\{\overline{\gamma_n}\}_{n\in\mathbb N}$ in  $PT^1(\Sigma_{g,n})$ such that,
$$\frac{v_8}{2}(i(\gamma_n,\gamma_n)-(2-2g))\leq\Vol(M_{\overline{\gamma_n}})< 8v_3 i(\gamma_n,\gamma_n),$$
where $v_3$ $(v_8)$ is the volume of the regular ideal tetrahedron (octahedron) and $i(\gamma_n,\gamma_n)$ the self-intersection number of $\gamma_n.$
 \end{repcorollary}

Similarly to \cite{Rod17}, given any geodesic multi-curve $\gamma$ and any continuous lift $\bar\gamma,$ one has a combinatorial lower bound for the volume of $M_{\bar\gamma}.$ Recall that a pants decomposition on an orbifold $\mathcal O,$  is a maximal family of disjoint simple closed geodesics on the underlying topological surface $\Sigma_\mathcal O$ which do not intersect  the singular points of $\mathcal O.$ We will show:

\begin{reptheorem}{1}
Given a pants decomposition $\Pi$ on a hyperbolic $2$-orbifold $\mathcal O$, a Seifert-fibered space $M$ over $\mathcal O,$ and a filling geodesic multi-curve $\gamma$ on $\mathcal O,$ for any closed continuous lift $\bar\gamma$ we have that\string: 

$$  \frac{v_3}{2}\sum_{P \in \Pi}(\sharp\{\mbox{isotopy classes of} \hspace{.2cm}  \bar\gamma\mbox{-arcs in} \hspace{.2cm} p^{-1}(P)\}-3)\leq\Vol(M_{\bar\gamma}),$$

where $v_3$ the volume of the regular ideal tetrahedron.
 \end{reptheorem}

Given a pair of pants $P,$ we say that two arcs $\bar\alpha,\bar\beta:[0,1] \rightarrow p^{-1}(P)$ with $\bar\alpha(\{0,1\})\cup \bar\beta(\{0,1\})\subset \partial( p^{-1}(P))$ are in the same isotopy class in $p^{-1}(P),$ if there exist an isotopy $\fun{h}{[0,1]_1\times[0,1]_2}{p^{-1}(P)}$ such that\string: $$h_0(t_2)=\bar\alpha(t_2),  \hspace{.2cm} h_1(t_2)=\bar\beta(t_2) \hspace{.2cm} \mbox{and} \hspace{.2cm}h([0,1]_1\times\{0,1\})\subset \partial( p^{-1}(P)).$$

 \begin{note}\label{bda} 
Up to isotopy, for a family of simple arcs without intersection there are only six configurations of arcs in $P$. These are shown in Figure \ref{badarcs}. The $3$ in the lower bound of Theorem \ref{1} comes from the fact that there are at most $3$ isotopy classes of $\bar\gamma$-arcs on $p^{-1}(P)$ projecting to such a configuration.
 \begin{figure}[h]
\centering
\includegraphics[scale=0.4] {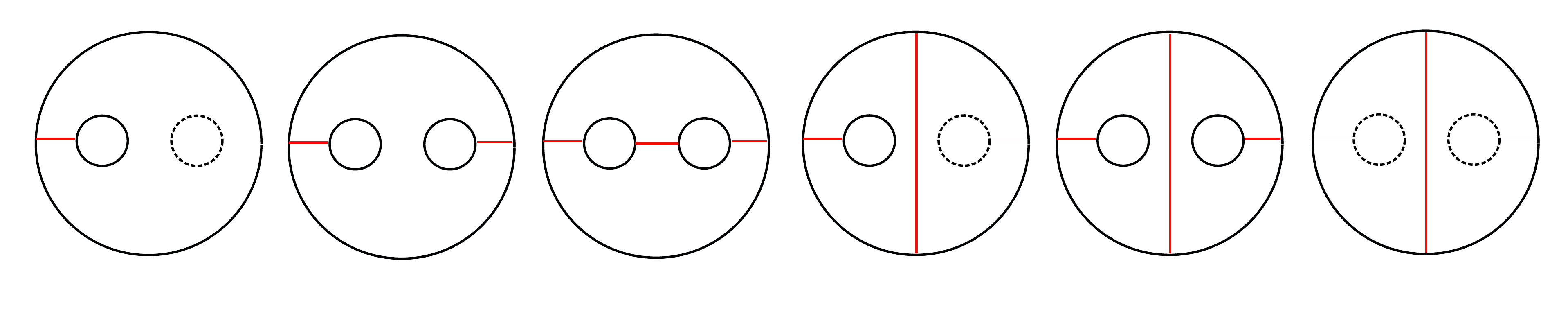}
\caption{ The projection on $P$ of the only six $\bar\gamma$-arcs configuration, up to isotopy, whose $\bar\gamma$-{arcs} project to pairwise disjoint simple arcs in $P.$ }\label{badarcs}
\end{figure} 
\end{note}

Before stating the main result to prove Theorem \ref{1}  we recall some definitions.
\vskip .2cm

If $N$ is a hyperbolic $3$-manifold and $S \subset N$ is an embedded incompressible surface, we will use $N\vert  S$ to denote the manifold  obtained from $N$ by cutting along $S$. The manifold $N\vert S$ is homeomorphic to the complement in $N$ of an open regular neighbourhood of $S.$ If one takes two copies of $N\vert S,$ and glues them along their boundary by using the identity diffeomorphism, one obtains the double of $N\vert S,$ which we denote by $D(N\vert S).$ 

\begin{definition}\label{DP}    Let $p:N\rightarrow \mathcal O$ be a Seifert-fibered space.
Let $P$ be a pair of pants belonging to a pant decomposition of a orbifold $\mathcal O$ and let $\gamma$ be a closed geodesic in $\mathcal O$ that is not isotopic into $P$. Moreover, assume that $P\cap\gamma$ is a finite set of geodesic arcs $\{\alpha_i\}_{i=1}^{n_P}$ connecting boundary components of $P.$ We define $P_{\bar\gamma}$ to be the set:

$$ p^{-1}(P)\setminus\bigcup_{i=1}^{n_P} \bar\alpha_i.$$

We also define $D({P}_{\bar\gamma}),$ as the gluing, via the identity homeomorphism, of two copies of $P_{\bar\gamma}$  along the punctured tori coming from:

$$\partial( p^{-1}(P))\setminus\bigcup_{i=1}^{n_P}  \bar\alpha_i.$$

Moreover, $D({P}_{\bar\gamma})$ is a link complement in the Seifert-fibered space $D(p^{-1}(P)),$ described as :

$$D(p^{-1}(P))\setminus \bigcup_{i=1}^{n_P}  D(\bar\alpha_i),$$

 where the projection orbifold of $D(p^{-1}(P)),$ whose underlying surface will be denoted by $ S^0,$ is one of the following:
 \begin{enumerate}
 \item either a genus two surface (if $\sharp(\partial(\Sigma_\mathcal O) \cap \partial P)=0)$;
 \item a surface of type $(1,2)$\footnote{By a surface of type $(n,m)$ we mean a genus $n$ surface with $m$ punctures.} (if $\sharp(\partial(\Sigma_\mathcal O)\cap \partial P)=1)$;
 \item a surface of type $(0,4)$ (if $\sharp(\partial(\Sigma_\mathcal O)\cap \partial P)=2).$ 
 \end{enumerate}
 
 Each $D(\bar\alpha_i)$ is a knot in $D(p^{-1}(P))$  obtained by gluing $\bar\alpha_i$ along the two points $\partial ( p^{-1}(P))\cap\bar\alpha_i$ via the identity.  See Figure \ref{piste}.
\end{definition}

\begin{figure}[h]
\includegraphics[scale=0.5] {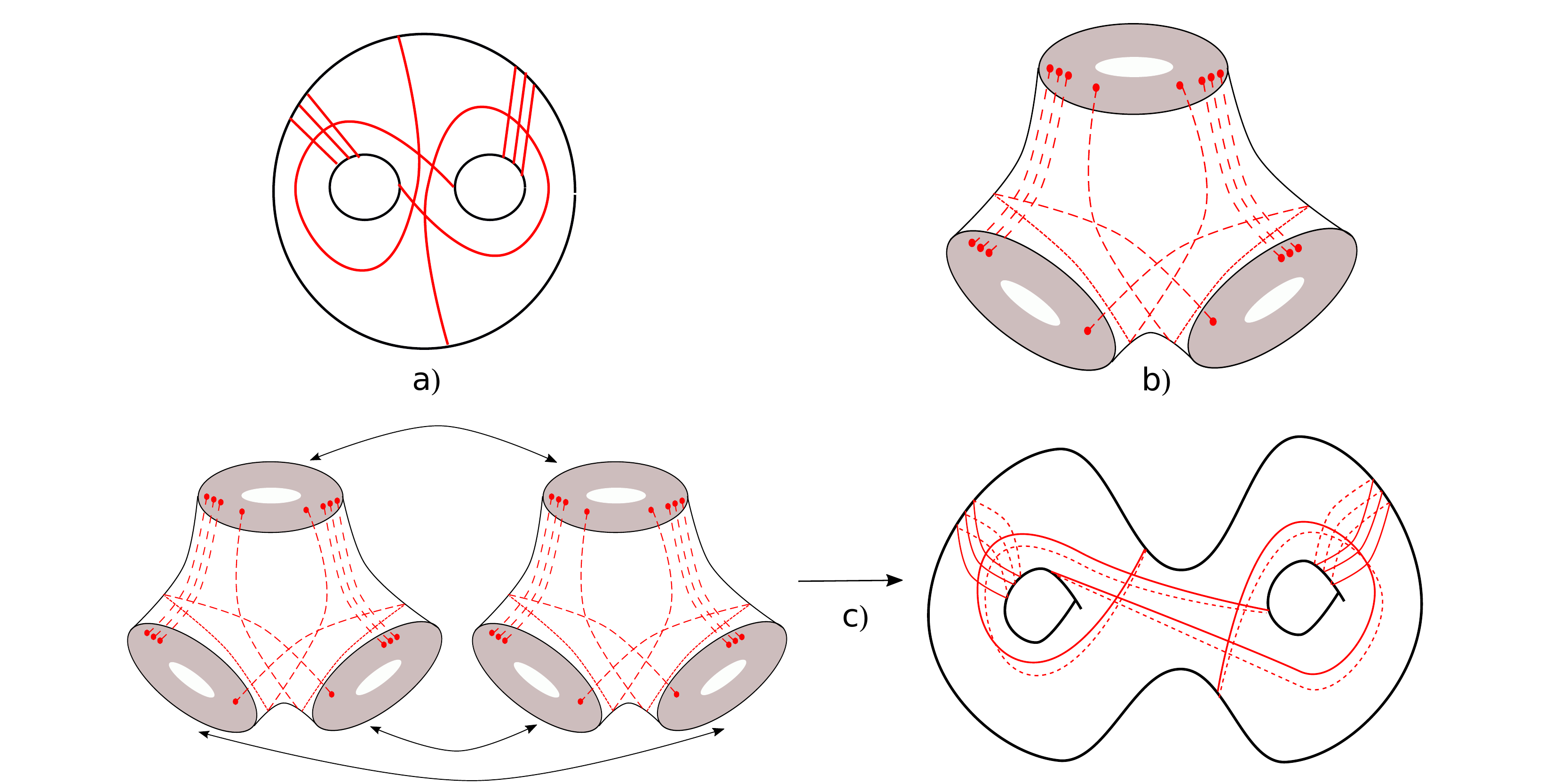}
\caption{$a)$ A pair of pants and a set of geodesic arcs connecting the boundary. $b)$  $P_{\bar\gamma}$ associated to $a),$ and  $c)$  $D({P}_{\overline{\gamma}})$ with the induced projection to $ S^0.$
}\label{piste}
\end{figure}

The key ingredient to prove Theorem \ref{1} is the following result due to Agol, Storm and Thurston, see \cite[ Theorem 9.1]{AST07}:

 \begin{theorem*}[Agol-Storm-Thurston]
 Let $N$ be a compact manifold with interior a hyperbolic $3$-manifold of finite volume. Let $S$ be a properly embedded incompressible surface in $N$, then:
 $$  \frac{v_3}{2} \| D(N\vert S)\|\leq\Vol(N)$$
  \end{theorem*}
 
 We now prove the lower bound for the volume of the canonical lift complement\string:
 
 \begin{proof}[\bf{Proof of Theorem \ref{1}}]
Let $\{\eta_i\}^{3g+n-3}_{i=1}$ be the simple closed geodesics inducing the pants decomposition $\Pi.$ Consider the  incompressible surface  $S\eqdef\bigsqcup^{3g+n-3}_{i=1} {(T_{\eta_i})_{\bar\gamma}}$ in $M_{\bar\gamma},$ where $(T_{\eta_i})_{\bar\gamma}$ is the incompressible punctured torus defined by $ p^{-1}(\eta_i))\setminus (p^{-1}(\eta_i)\cap\bar\gamma)$, see \cite[Lemma 2.5]{Rod17}. From \cite[Theorem 9.1]{AST07} we deduce that\string:
$$  \frac{v_3}{2}\sum_{P \in \Pi} \| D({P}_{\bar\gamma})\|=\frac{v_3}{2} \| D(M_{\bar\gamma}\vert S)\|\leq\Vol(M_{\bar\gamma})$$
For each pair of pants $P$ we have\string:
$$v_3\sharp\{\mbox{cusps of}  \hspace{.2cm} D({P}_{\bar\gamma})^{hyp}\} \leq  \Vol( D({P}_{\bar\gamma})^{hyp})\leq v_3\| D({P}_{\bar\gamma})^{hyp}\|  = v_3\| D({P}_{\bar\gamma})\|$$
 where $D({P}_{\bar\gamma})^{hyp}$ is the atoroidal piece of $D({P}_{\bar\gamma}),$   i.e., the complement of the characteristic sub-manifold, with respect to its JSJ-decomposition. The first and second inequality come from \cite{Ada88}  and \cite{Gro82} respectively. 
 \vskip .2cm
 Let $\Omega$ be the subset of $\gamma$-arcs on $P$ having one arc for each isotopy class of $\bar\gamma$-arcs on $p^{-1}(P)$. This means that $D({P}_{\bar\gamma})^{hyp}\cong D({P}_{\overline{\Omega}})^{hyp}.$ Moreover, $D({P}_{\overline{ \Omega}})$ can be seen as a link complement in $D(p^{-1}(P))$, see Definition \ref{DP}, whose projection to $S^0$ is a union of closed loops transversally homotopic to a union closed loops in minimal position. The atoroidal piece of $D({P}_{\overline{\Omega}})$ corresponds to the subsurface of $S^0$ which $D(\Omega)$ fills (Theorem \ref{sf}).

  \begin{enumerate}
   \item If the $\Omega$-arc configuration on $P$ is in the list of Remark \ref{bda}, then by Theorem \ref{sf} we have that $D({P}_{\bar\gamma})^{hyp}=\emptyset$  and Remark \ref{bda} also gives us:
 $$\frac{v_3}{2}(\sharp\{\mbox{isotopy classes of} \hspace{.2cm}  \bar\gamma\mbox{-arcs in} \hspace{.2cm} p^{-1}(P)\}-3)\leq v_3\sharp\{\mbox{cusps of}  \hspace{.2cm} D({P}_{\bar\gamma})^{hyp}\}.$$
 \item If the $\Omega$-arc configuration on $P$ is not in the list of Remark \ref{bda}, then there is at least one geometric intersection point on the projection of the link complement $D({P}_{\overline{ \Omega}})$  to $S^0.$  \end{enumerate} 
\vskip .2cm

\begin{figure}[h]
\centering
\includegraphics[scale=0.4] {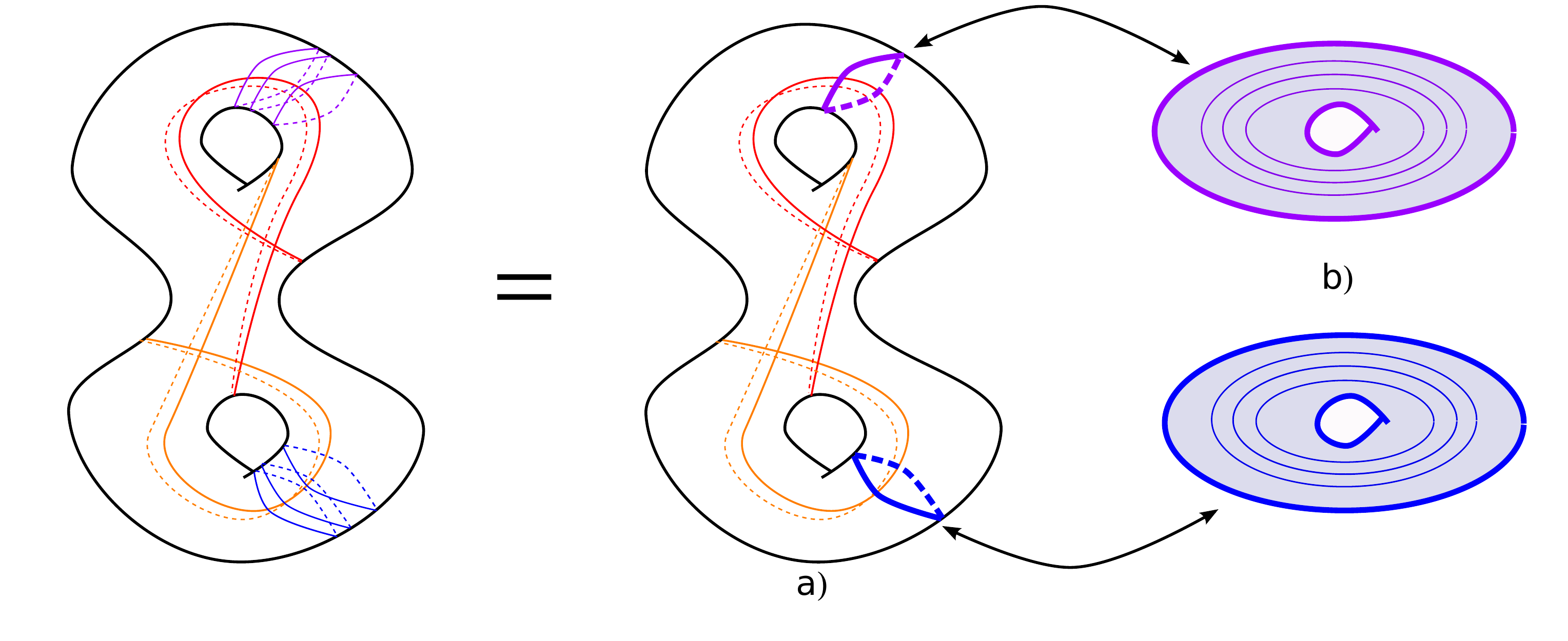}
\caption{The JSJ-decomposition of $D({P}_{\overline{\gamma}})$ of Figure \ref{piste}.c.
}\label{DPgJSJ}
\end{figure}
By Theorem \ref{sf} we conclude that $D({P}_{\bar\gamma})^{hyp}\neq\emptyset.$ We will now define an injective function: 
$$\left\{ \begin{array}{l} \bar\gamma\mbox{-arcs in} \\ \hspace{.5cm} p^{-1}(P)\end{array}\right\} \overset{\varphi}\longrightarrow \left\{ \begin{array}{l}\hspace{.1cm} \mbox{cusps of} \\ D({P}_{\bar\gamma})^{hyp}\end{array}\right\}$$
where the target can be decomposed as:
$$\left\{ \begin{array}{l}\hspace{.1cm} \mbox{cusps of} \\ D({P}_{\bar\gamma})^{hyp}\end{array}\right\}=\left\{ \begin{array}{l}\hspace{.3cm}\mbox{splitting tori of the} \\ \mbox{JSJ-decomposition of}  \\ \hspace{1.3cm} D({P}_{\bar\gamma})\end{array}\right\}\amalg\left\{ \begin{array}{l}\hspace{.8cm}\mbox{cusp in}  \\   D({P}_{\bar\gamma}) \cap  D({P}_{\bar\gamma})^{hyp}\end{array}\right\}$$ 
The function $\varphi$ is defined as follows: if the cusps in $D({P}_{\bar\gamma})$ are induced by the $\bar\gamma$-arc in $p^{-1}(P)$ belonging to the characteristic sub-manifold of $D({P}_{\bar\gamma}),$  $\varphi$ maps it to a splitting tori connecting the hyperbolic piece with the component of the characteristic sub-manifold where it is contained. Otherwise, the cusp belongs to $D({P}_{\bar\gamma})^{hyp}$ and $\varphi$ sends it to itself, see Figure \ref{DPgJSJ}. 
Assume that there are more isotopy classes of $\bar\gamma$-arcs in $p^{-1}(P)$ than the number of cusps of $D({P}_{\bar\gamma})^{hyp}$. Then, there are two tori, associated with non-isotopic $\bar\gamma$-arcs in $p^{-1}(P),$ that belong to the same connected component of the characteristic sub-manifold. Since each component of the characteristic sub-manifold is a Seifert-fibered space over a punctured surface we have that all such arcs correspond to regular fibres. Thus, they are isotopic in the corresponding component hence isotopic in $p^{-1}(P),$ contradicting the fact that they were not isotopic.    \end{proof}
This result implies that there exist a filling geodesic multi-curve $\gamma$ on $\mathcal O$ with bounded components such that $\Vol(M_{\bar\gamma})$ can be as large as we want. Let us fix a pants decomposition on $\mathcal O,$ then for any $N\in\mathbb{N}$ there exist a closed geodesic with at least $N$ homotopy classes of geodesic arcs in one pair of pants. This is constructed by taking $N$ non-homotopic geodesic arcs in a pair of pants and linking them to form a filling geodesic multi-curve on $\mathcal O.$
\vskip .2cm
The lower  bound of the volume of  $M_{\bar\gamma}$ obtained in Theorem \ref{1} does not have control on the length of the geodesic multi-curve, even if each homotopy class of $\gamma$-arcs contributes to the length of $\gamma.$ 

\begin{question}
Given a hyperbolic orbifold, estimate the volume of $M_{\bar\gamma}$ among the filling geodesic multi-curve $\gamma$ whose length  is bounded by a fixed constant. 
\end{question}

\nocite{BP1992,CEM2006,Th1978,MT1998}

\thispagestyle{empty}
{\small
\markboth{References}{References}
\bibliographystyle{alpha} 
\bibliography{mybib}{}

}
	\bigskip

\noindent Department of Mathematics, Boston College.

\noindent 140 Commonwealth Avenue Chestnut Hill, MA 02467.

\noindent Maloney Hall
\newline \noindent
email: \texttt{cremasch@bc.edu}
\bigskip

\noindent Department of Mathematics and Statistics,  University of Helsinki.  
\newline\noindent Pietari Kalminkatu 5,  Helsinki FI 00014.
\newline \noindent 
 email: \texttt{jose.rodriguezmigueles@helsinki.fi}

									\end{document}